\documentclass[11pt]{amsart}

\usepackage{amssymb}
\usepackage{amsmath}
\usepackage{amsthm}
\usepackage{color} 
\usepackage{verbatim}
\usepackage{bbm}
\usepackage{dsfont}
\usepackage{mathtools}

\newtheorem{theorem}{Theorem}[section]
\newtheorem{lemma}{Lemma}[section]
\newtheorem{remark}{Remark}[section]
\newtheorem{assumption}{Assumption}[section]

\numberwithin{equation}{section}

\newcommand{\Rd}{\mathbb{R}^d}
\newcommand{\ve}{\varepsilon}

\newcommand{\B}{\mathrm B}

\newcommand{\px}{\psi_{(\xi)}}

\newcommand{\pxt}{\psi_{(\xi_t)}}

\newcommand{\psk}{\psi_{(\sigma_k)}}

\newcommand{\pxop}{\frac{\psi_{(\xi)}}{\psi}}

\newcommand{\pxsop}{\frac{\psi_{(\xi)}^2}{\psi}}
\newcommand{\pxtsop}{\frac{\psi_{(\xi_t)}^2}{\psi}}
\newcommand{\pxsops}{\frac{\psi_{(\xi)}^2}{\psi^2}}
\newcommand{\pxtsops}{\frac{\psi_{(\xi_t)}^2}{\psi^2}}

\newcommand{\vxosbw}{\frac{|v_{(\xi)}(x)|}{\sqrt{\mathrm{B}_1(x,\xi)}}}
\newcommand{\vxosbt}{\frac{|v_{(\xi)}(x)|}{\sqrt{\mathrm{B}_2(x,\xi)}}}
\newcommand{\vzosbw}{\frac{|v_{(\zeta)}(y)|}{\sqrt{\mathrm{B}_1(y,\zeta)}}}
\newcommand{\vzosbt}{\frac{|v_{(\zeta)}(y)|}{\sqrt{\mathrm{B}_2(y,\zeta)}}}

\newcommand{\sfa}{\sup_{\alpha\in\mathfrak A}}

\begin{document}

\title[Regularity of nonlinear degenerate elliptic PDEs]{Interior regularity of fully nonlinear degenerate elliptic equations, I: Bellman equations with constant coefficients}
\author{Wei Zhou}
\address{School of Mathematics, University of Minnesota}
\email{zhoux123@math.umn.edu}


\begin{abstract}
\noindent
This is the first of a series of papers on the interior regularity of fully nonlinear degenerate elliptic equations of second order. We consider here a stochastic optimal control problem in a domain, in which the diffusion coefficients, drift coefficients and discount factor are independent of the spatial variables. Under appropriate assumptions, for $k=0,1$, when the terminal and running payoffs are globally $C^{k,1}$, we establish the interior $C^{k,1}$-smoothness of the value function, which yields the existence and uniqueness of the solution to the associated Dirichlet problem for the possibly degenerate Bellman equation with constant coefficients. Interior estimates for first and second derivatives of the solution are also obtained. The results are applicable to fully nonlinear degenerate elliptic equations in the form of $F(u_{x^ix^j}(x),x)=0$ which are invariant under the action of the orthogonal group on the Hessian matrix, including Monge-Amp\`ere equations and other Hessian equations under suitable settings, as discussed in subsequent papers.

\end{abstract}

\maketitle

\section{Introduction}

This paper is concerned with the interior $C^{1,1}$-regularity theory for the degenerate Bellman equation with constant coefficients. We are motivated  by \cite{MR1211724}, in which interior $C^{0,1}$-regularity result was obtained.

We consider the time-homogeneous stochastic optimal control problem in a domain. Given a family of controlled diffusion processes governed by It\^o stochastic equations:
\begin{equation*}
x_t^{\alpha,x}=x+\int_0^t\sigma^{\alpha_s}(x_s^{\alpha,x})dw_s+\int_0^tb^{\alpha_s}(x_s^{\alpha,x})ds,
\end{equation*}
where $w_t$ is a $d_1$-dimensional Wiener process, the associated time-homogeneous stochastic optimal control problem in a domain has the value function
\begin{equation*}
v(x)=\sfa E\bigg[g\big(x^{\alpha,x}_{\tau^{\alpha,x}}\big)e^{-\phi^{\alpha,x}_{\tau^{\alpha,x}}}+\int_0^{\tau^{\alpha,x}}f^{\alpha_t}\big(x_t^{\alpha,x}\big)e^{-\phi_t^{\alpha,x}}dt\bigg],
\end{equation*}
with
$$\phi_t^{\alpha,x}=\int_0^t c^{\alpha_s}(x_s^{\alpha,x})ds,$$
where $\mathfrak A$ is the set of policies, and for each $\alpha\in\mathfrak A$, $\tau^{\alpha,x}$ is the first exit time of $x_t^{\alpha,x}$ from the domain $D\subset\Rd$ ($d\ge2$), the nonnegative function $c^{\alpha}$ is the discount factor and the functions $f^{\alpha}$ and $g$ are the running payoff and terminal payoff respectively. The associated dynamic programming equation is the possibly degenerate Bellman equation with Dirichlet boundary condition:
\begin{equation}\tag{\bf{B}}\label{Bellmanequation}
\left\{
\begin{array}{rcll}
\displaystyle\sup_{\alpha\in A}\big[(a^\alpha)_{ij}u_{x^ix^j}+ (b^\alpha)_iu_{x^i}-c^\alpha u+f^\alpha\big]&=&0 &\text{in } D\qquad\\
u&=&g &\text{on }\partial D,
\end{array}
\right. 
\end{equation}
where $A$ is the control set, the matrix $a^\alpha=[(a^\alpha)_{ij}]_{d\times d}=(1/2)\sigma^\alpha(\sigma^\alpha)^*$  for each $\alpha\in A$, and summation convention of repeated indices is assumed.

If the value function $v$ is in the class of $\in C^{2}(D)\cap C^0(\bar D)$, then $v$ is a classical solution to (\ref{Bellmanequation}) due to Bellman principle and It\^o's formula. However, in general, $v$ is not sufficiently smooth to satisfy (\ref{Bellmanequation}). An interesting problem is establishing sufficient conditions under which $v$ has derivatives up to and including second order and uniquely solves (\ref{Bellmanequation}). Both PDE theoretic and probabilistic methods have been utilized in previous literature, see, e.g.,  \cite{MR678347, MR688919, MR992979, MR1211724, MR725360, MR765302, MR984219}. For PDE theoretic approach, the difficulties contain the degeneracy and fully nonlinearity of the elliptic equation. For probabilistic approach, the difficulties include the randomness and infiniteness of the exit time and the non-vanishing terminal payoff.

In this work, we restrict our attention to the problem in which the diffusion coefficient $\sigma^\alpha$, drift coefficient $b^\alpha$ and discount factor $c^\alpha$ are all independent of spatial variables, so that the associated Bellman equation is with constant coefficients. The main reason is that convex fully nonlinear elliptic equations in the form of
\begin{equation*}\label{FF}
F\big(u_{x^ix^j}(x), x\big)=0\\
\end{equation*}
can be rewritten as Bellman equations with constant coefficients, see \cite{MR1284912}. The other reason is that in \cite{MR3047001}, we obtained $C^{1,1}$-smoothness of $v$ and unique solvability of (\ref{Bellmanequation}) for non-constant coefficients, under the assumption of nondegeneracy of the diffusion coefficients along the normal to the boundary. Therefore, we are interested in obtaining the same smoothness results without this assumption for simpler equations. Instead, our main assumptions here are Assumptions \ref{a1} and \ref{a2}, which are general enough to make our theorems applicable to fully nonlinear degenerate elliptic equations in the form of $F(u_{x^ix^j}(x),x)=0$ which are invariant under orthogonal congruence on the Hessian matrix, including Monge-Amp\`ere equations, as studied in the sequent paper \cite{InteriorRegularityII}.

Our main results are the following: under Assumptions \ref{a1} and \ref{a2},
\begin{itemize}
\item If $f^\alpha,g\in C^{0,1}(\bar D)$, then $v\in C^{0,1}_{loc}(D)\cap C^0(\bar D)$. 
\item If $f^\alpha\in C^{0,1}(\bar D)$, $g\in C^{1,1}(\bar D)$ and $f^\alpha+K|x|^2$ is convex in $D$ for some constant $K$, then $v$ is convex after adding the function given in (\ref{e2}). 
\item If we further assume the weak nondegenacy of the diffusion term, see Remark \ref{wnc}, then $v\in C^{1,1}_{loc}(D)\cap C^{0,1}(\bar D)$, and (\ref{Bellmanequation}) is uniquely solved by $v$ in this function space.
\item Interior estimates of the first and second derivatives of $v$ are given by (\ref{e1}) and (\ref{e3}), under respective regularity assumptions on $f^\alpha$ and $g$.
\end{itemize}

Our interior $C^{0,1}$-regularity result is a non-essential generalization of the corresponding result in \cite{MR1211724}, in the sense of allowing $b^\alpha$ and $c^\alpha$ nonvanishing. (Note that we don't assume that $c^\alpha$ has a positive lower bound.) The interior $C^{1,1}$-regularity result is totally new. It is worth emphasizing that the $C^{k,1}$-regularity of the boundary data doesn't ensure the same global regularity for the solution of the Dirichlet problem in general. More precisely, if the boundary data $g$ is $C^{k,1}$ on $\partial D$, $v$ may not be $C^{k,1}$ up to the boundary, by even considering the Wiener process and the associated heat equation or Laplace's equation. Instead, the best regularity on $v$ we may expect is the interior $C^{k,1}$-regularity. In this sense, our regularity results on $v$ are optimal. We also provide interior estimates of first and second derivatives. We show that when $v\in C^{0,1}_{loc}(D)$, its $C^{0,1}$-norm doesn't blow up faster than $1/\operatorname{dist}(\cdot,\partial D)$, which is sharp due to Example 4.1.1 in \cite{MR2144644}, and when $v\in C^{1,1}_{loc}(D)$, its $C^{1,1}$-norm doesn't blow up faster than $1/\operatorname{dist}(\cdot,\partial D)^2$, whose sharpness is unknown by the author.

Unlike \cite{MR1211724}, we write down the entire paper in probabilistic terms rather than PDE terms, in order to show the ideas more intuitively and express several quantities by explicit formulas. We admit that in some circumstances, using PDE terms is more economical as far as computations and assumptions are concerned. However, we believe that the entire paper can be translated into a pure analysis of PDE paper like \cite{MR1211724}.

Our main theorems are stated in Section \ref{smt}. The online of the remaining sections concerning the proof is discussed in Section \ref{osp}.

Throughout the article, the summation convention for repeated indices is assumed, even when both repeated indices appear in the superscript. We usually put the indices in the superscript, since the subscript is for the time variable of stochastic processes. 
Given any sufficiently smooth function $u(x)$ from $\Rd$ to $\mathbb R$, for $y,z\in\Rd$, let
\begin{align*}
u_{(y)}=& u_{x^i}y^i,\qquad  u_{(y)(z)}= u_{x^i x^j}y^i z^j,\qquad u_{(y)}^2=(u_{(y)})^2.
\end{align*}
We denote the gradient vector of $u$ by $u_x$ and the Hessian matrix of $u$ by $u_{xx}$.
For any matrix $\sigma=(\sigma^{ij})$, $\sigma^*$ represents its tranpose and $\|\sigma\|^2:=\mathrm{tr}(\sigma\sigma^*).$
We also define
\[s\wedge t=\min\{s,t\},\qquad  s \vee t=\max\{s,t\}.\]
Constants appearing in inequalities are usually not indexed. They may differ even in the same chain of inequalities.


\section{Statement of main theorems}\label{smt}
Let $d$ and $d_1$ be integers and $A$ be a separable metric space. Assume that the following continuous and bounded functions on $A$ are given:
\begin{itemize}
\item $d\times d_1$ matrix-valued function $\sigma^\alpha=(\sigma^\alpha_1,...,\sigma^\alpha_{d_1})$,
\item $\Rd$-valued function $b^\alpha$,
\item real-valued non-negative function $c^\alpha$.
\end{itemize}
Let $(\Omega,\mathcal{F},P)$ be a complete probability space, $\{\mathcal{F}_t;t \ge 0\}$ be an increasing filtration of $\sigma $-algebras $\mathcal{F}_t \subset \mathcal{F}$ which are complete with respect to $(\mathcal{F},P)$, and $(w_t ,\mathcal{F}_t ;t \ge 0)$ be a $d_1$-dimensional Wiener process on $(\Omega,\mathcal{F},P)$. Denote by $\mathfrak{A}$  the set of progressively measurable $A$-valued processes $\alpha_t=\alpha_t(\omega)$. 

Let $D$ be a $C^{3}$ bounded domain in $\Rd$ described by a $C^{3}$ real-valued function $\psi$ which is non-singular on $\partial D$, i.e.
\begin{equation}\label{domain}
D:=\{x\in\Rd:\psi(x)>0\},\qquad|\psi_x|\ge1\mbox{ on }\partial D.
\end{equation}
For the sake of simplicity in the statement of the results and their proofs, we suppose that
$$|\psi|_{3,D},\|\sigma\|_{0,A},|b|_{0,A}, |c|_{0,A}\le K_0,$$
where $K_0\in[1,\infty)$ is constant.

In the domain $D$, a real-valued function $g(x)$ is given, which is bounded and Borel measurable. On the set $A\times D$, a real-valued function $f^\alpha(x)$ is defined, which is bounded and Borel measurable in $A\times D$.

Now we consider the stochastic optimal control of degenerate diffusion processes in which $D$ is the domain, $A$ is the control set, $\mathfrak A$ is the set of policies, $\sigma^\alpha$, $b^\alpha$, $c^\alpha$ are diffusion, drift and discount coefficients, and $f^\alpha(x)$, $g(x)$ are running payoff and terminal payoff, respectively. To be precise, for each $\alpha_t\in\mathfrak{A}$ and $x\in D$, the degenerate diffusion process is given by 
\begin{equation}\label{itox}
x_t^{\alpha,x}=x+\int_0^t\sigma^{\alpha_s}dw_s+\int_0^tb^{\alpha_s}ds.
\end{equation}
The value function of the stochastic optimal control is known as 
\begin{equation}
v(x)=\sup_{\alpha\in\mathfrak{A}}E\bigg[g\big(x^{\alpha,x}_{\tau^{\alpha,x}}\big)e^{-\phi^{\alpha}_{\tau^{\alpha,x}}}+\int_0^{\tau^{\alpha,x}}f^{\alpha_s}\big(x_s^{\alpha,x}\big)e^{-\phi^{\alpha}_s}ds\bigg],\label{v}
\end{equation}
where for each $\alpha\in \mathfrak A$ and $t\ge 0$,
\begin{equation}\label{phidiscount}
\phi_t^{\alpha}:=\int_0^tc^{\alpha_s}ds,
\end{equation}
and for each $\alpha\in\mathfrak A$ and $x\in D$, $\tau^{\alpha,x}$ is the first exit time of $x_t^{\alpha,x}$ from $D$, namely, $\tau^{\alpha,x}:=\inf\{t\ge0:x_t^{\alpha,x}\notin D\}.$

From now on, we use the common abbreviated notation, according to which we put the superscripts $\alpha$ and $x$ beside the expectation sign instead of explicitly exhibiting them inside the expectation sign for every object that can carry all or part of them. Namely, 
\begin{eqnarray*}
\lefteqn{E^\alpha_x\bigg[g\big(x_\tau\big)e^{-\phi_{\tau}}+\int_0^{\tau}f^{\alpha_s}\big(x_s\big)e^{-\phi_s}ds\bigg]}\\
&:\displaystyle=E\bigg[g\big(x^{\alpha,x}_{\tau^{\alpha,x}}\big)e^{-\phi^{\alpha}_{\tau^{\alpha,x}}}+\int_0^{\tau^{\alpha,x}}f^{\alpha_s}\big(x_s^{\alpha,x}\big)e^{-\phi^{\alpha}_s}ds\bigg].
\end{eqnarray*}

We also denote by $\mathbb S^d$ (resp. $\mathbb O^d$) the set of $d\times d$ symmetric (resp. orthogonal) matrices and introduce
\begin{equation}\label{condition}
\mu(\xi)=\inf_{\zeta: (\xi,\zeta)=1}\sup_{\alpha\in A}(a^{\alpha})_{ij}\zeta^i\zeta^j,  \mbox{ with } a^\alpha=(1/2)\sigma^\alpha(\sigma^\alpha)^*,
\end{equation}
\begin{equation}\label{conditions}
\mu=\inf_{|\xi|=1}\mu(\xi).
\end{equation}

Our assumptions and theorems are the following:

\begin{assumption}\label{a1} For each $x\in D$, we have 
\begin{equation}
\displaystyle\sup_{\alpha\in A}L^\alpha\psi(x)\le-1,
\mbox{ where }
\displaystyle L^\alpha =(a^\alpha)_{ij} \frac{\partial^2}{\partial x^i\partial x^j}+(b^\alpha)_i \frac{\partial}{\partial x^i}.
\end{equation}

\end{assumption}

\begin{assumption}\label{a2}
For each $q\in\mathbb O^d$ and $(\gamma,p,z,x)\in\mathbb S^d\times \Rd\times\mathbb R\times D$,
\begin{equation}
\begin{gathered}
\sup_{q\in\mathbb O^d}\sup_{\alpha\in  A}\big[(qa^\alpha q^*)_{ij}\gamma^{ij}+(b^\alpha)_ip^i-c^\alpha z+f^\alpha(x)\big]\\
=\sup_{\alpha\in A}\big[(a^\alpha)_{ij}\gamma^{ij}+(b^\alpha)_ip^i-c^\alpha z+f^\alpha(x)\big].
\end{gathered}
\end{equation}
\end{assumption}

\begin{remark}
For example, if the set $\mathbb{A}:=\{a^\alpha: \alpha\in A\}$ is $\mathbb O^d$-invariant, i.e., for any orthogonal matrix $q\in\mathbb O^d$, $q\mathbb A q^*=\mathbb A$, and the following conditions
$$b^\alpha=\tilde b(a^\alpha)=\tilde b(qa^\alpha q^*), \qquad c^\alpha=\tilde c(a^\alpha)=\tilde c(qa^\alpha q^*),$$
$$f^\alpha(x)=\tilde f(a^\alpha,x)=\tilde f(qa^\alpha q^*,x)$$
hold for each $(q,\alpha)\in \mathbb O^d\times A$, then 
$$\sup_{\alpha\in  A}\big[(qa^\alpha q^*)_{ij}\gamma^{ij}+(b^\alpha)_ip^i-c^\alpha z+f^\alpha(x)\big]
=\sup_{\alpha\in A}\big[(a^\alpha)_{ij}\gamma^{ij}+(b^\alpha)_ip^i-c^\alpha z+f^\alpha(x)\big],$$
for each $q\in\mathbb O^d$, and consequently Assumption \ref{a2} holds.

\end{remark}

\begin{theorem}\label{t1}
Under Assumption \ref{a1}, the value function $v$ given by (\ref{v}) is well-defined, and we have
\begin{equation}
|v(x)|\le|g|_{0,\partial D}+\psi(x)\sup_{\alpha\in A}|f^\alpha|_{0,D},\ \forall x\in D.
\end{equation}
\end{theorem}

\begin{theorem}\label{t2}
Under Assumptions \ref{a1} and \ref{a2}, if $f^\alpha\in C^{0,1}(\bar D)$, $g\in C^{0,1}(\partial D)$, and
$$\sup_{\alpha\in A}|f^\alpha|_{0,1,D}, |g|_{0,1, \partial D}\le K_0,$$
then $v\in C^{0,1}_{loc}(D)\cap C(\bar D)$, and for a.e. $x\in D$, we have
\begin{equation}\label{e1}
\big|v_{(\xi)}\big|\le N\bigg(|\xi|+\frac{|\psi_{(\xi)}|}{\psi^{1/2}}\bigg),\ \forall\xi\in\Rd,
\end{equation}
where $N=N(K_0, d, d_1, D)$ is constant.
\end{theorem}
\begin{theorem}\label{t3}
Under Assumptions \ref{a1} and \ref{a2}, if $f^\alpha\in C^{0,1}(\bar{D})$, $g\in C^{1,1}(\partial {D})$,
$$\sup_{\alpha\in A}|f^\alpha|_{0,1,D}, |g|_{1,1,D}\le K_0,$$
and for each $\alpha\in A$, $f^\alpha+K_0|x|^2$ is convex , then for each constant $\kappa>0$, the function
\begin{equation}\label{e2}
v+N\Big[|x|^2+\psi\Big(\log\frac{\psi}{\kappa}-1\Big)\Big]
\end{equation}
is convex in the set $\{x\in D:\psi(x)\le\kappa\}$,
where $N=N(K_0, d, d_1, D)$ is constant.

\allowdisplaybreaks
If we additionally assume that $\mu>0$, then $v\in C^{1,1}_{loc}(D)\cap C^{0,1}(\bar D)$, and for a.e. $x\in D$, we have
\begin{equation}\label{e3}
-N\bigg(|\xi|^2+\frac{\psi_{(\xi)}^2}{\psi}\bigg)\le v_{(\xi)(\xi)}\le \mu(\xi/|\xi|)^{-1}N\frac{|\xi|^2}{\psi},\ \forall\xi\in\Rd,
\end{equation}
where $N=N(K_0, d, d_1, D)$ is constant. Furthermore, $v$ is the unique solution in $C^{1,1}_{loc}(D)\cap C^{0,1}(\bar D)$ of the Dirichlet problem
\begin{equation}
\left\{
\begin{array}{rcll}
\displaystyle\sup_{\alpha\in A}\big[L^\alpha v-c^\alpha v+f^\alpha\big]&=&0 &\text{a.e. in } D\\
v&=&g &\text{on }\partial D.
\end{array}
\right.  \label{bellmanae}
\end{equation}
\end{theorem}

\begin{remark}\label{wnc}
The condition $\mu(\xi)>0$ means that the term $v_{(\xi)(\xi)}$ essentially appear in the Bellman equation in (\ref{bellmanae}). 
It is also not hard to see that
$$\mu=\inf_{|\zeta|=1}\sup_{\alpha\in A}(a^\alpha)_{ij}\zeta^i\zeta^j.$$
Note that the condition $\mu>0$ is called ``weak nondegeneracy condition" in some previous literature,  which holds if and only if for any $\zeta\ne0$, there exists an element in the control set $A$, such that the corresponding diffusion term $a^\alpha$ is nondegenerate in the direction of $\zeta$.
\end{remark}

\begin{remark}
The first derivative estimate (\ref{e1}) is sharp due to Example 4.1.1 in \cite{MR2144644}. The author doesn't know whether the second derivative estimate (\ref{e3}) is sharp.
\end{remark}

\section{Outline and Strategy of the proof of Theorems {\ref{t1}}-{\ref{t3}}}\label{osp}
In Section \ref{sub1} we use Assumption \ref{a1} to prove Theorem \ref{t1}. 

To prove Theorems \ref{t2} and \ref{t3}, we first reduce the original problem of showing the existence of generalized derivatives to a priori estimate of the derivatives, which is explained in Section \ref{sub2}.

To estimate the derivatives, we differentiate both sides of the probabilistic representation (\ref{v}). The main difficulty comes from the non-vanishing terminal payoff and the random  unbounded exit time of the diffusion processes. Thus for simplicity in discussing the strategy we temporarily let $c^\alpha=f^\alpha=0$. Heuristically, utilizing Bellman principle and then differentiating $v$ in the direction of $\xi$, we wish to have
\begin{align}
v_{(\xi)}(x)&\le\sfa E_x^\alpha\big[v_{(\xi_\tau)}(x_\tau)\big],\label{firstder}\\
v_{(\xi)(\xi)}(x)&\ge\sfa E_x^\alpha\big[v_{(\xi_\tau)(\xi_\tau)}(x_\tau)+v_{(\eta_\tau)}(x_\tau)\big],\label{secondder}
\end{align}
where $\xi_t^{\alpha,\xi}$ and $\eta_t^{\alpha,\eta}$ should be the first and second derivatives of the state process $x_t^{\alpha,x}$ with respect to its initial position in some sense. For this reason, in Section \ref{sub3}, we introduce the quasiderivatives which are more general than the traditional derivatives of stochastic processes and can somehow fit in the expectations on the right-hand side of (\ref{firstder}) and (\ref{secondder}). 

We hope that $\xi_\tau^{\alpha,\xi}$ is tangent to the boundary, so that we can replace $v$ in the leading term in the expectations in (\ref{firstder}) and (\ref{secondder}) with $g$. Therefore, in Section \ref{sub4}, we seek such quasiderivatives by choosing appropriate parameters in their expressions. Note that since the diffusion processes are random, we have no way to figure out when or where they will exit the domain. Thus it is not an easy task to make the quasiderivatives always tangent to the boundary when the diffusion processes exit the domain. With the help of two nonnegative local supermartingales, we are able to show that our first quasiderivatives are tangent to the boundary when the diffusion processes exit the domain almost surely. 

Gathering these auxiliary tools and results, we prove Theorem \ref{t2} and Theorem \ref{t3} in Sections \ref{sub5} and \ref{sub6}, respectively. More precisely,  after establishing (\ref{firstder}), by replacing $\xi$ with $-\xi$, we obtain the first derivative estimate. As far as the second derivatives are concerned, we notice that
$$4v_{(\xi)(\eta)}=v_{(\xi+\eta)(\xi+\eta)}-v_{(\xi-\eta)(\xi-\eta)},$$
so it suffices to estimate $v_{(\xi)(\xi)}$. From (\ref{secondder}) we can just get the second derivative estimate from below. To obtain the second derivative estimate from above, we make use of the associated Bellman equation under the assumption of weak nondegenercy. The existence result is known, and the uniqueness result is a corollary of a theorem in time-inhomogeneous case.


\section{Proof of Theorem {\ref{t1}}}\label{sub1}
Theorem \ref{t1} is a direction conclusion from the following lemma, which says that the moments of the exit times are uniformly bounded under Assumption \ref{a1}.

\begin{lemma}\label{taun}
If Assumption \ref{a1} holds, then for any $x\in D$,
\begin{align}
\sfa E^\alpha_x\tau^n\le n!|\psi|^{n-1}_{0,D}\psi(x), \ \forall n\in\mathbb N.
\end{align}
\end{lemma}

\begin{proof}
It suffices to prove the inequality for each $\alpha\in\mathfrak A$ and notice that
\allowdisplaybreaks\begin{align*}
E^\alpha_x\tau\le& -E^\alpha_x\int_0^{\tau}L\psi dt=\psi(x)-E^\alpha_x\psi(x_{\tau})=\psi(x),\\
E^\alpha_x\tau^n=&nE_x^\alpha\int_0^\infty(\tau-t)^{n-1}\mathbbm 1_{\tau>t}dt=nE\int_0^\infty \mathbbm 1_{\tau>t}E(\tau^{\alpha,x_t})^{n-1}dt\\
\le& n\sup_{y\in D}E^\alpha_y\tau^{n-1}\cdot E^\alpha_x\tau\le n\sup_{y\in D}E^\alpha_y\tau^{n-1}\cdot\psi(x).
\end{align*}
\end{proof}

\begin{proof}[Proof of Theorem \ref{a1}] Notice that
$$|v(x)|\le|g|_{0,\partial D}+\sup_{\alpha\in A}|f^\alpha|_{0,D}\sfa E_x^\alpha\tau\le |g|_{0,\partial D}+\psi(x)\sup_{\alpha\in A}|f^\alpha|.$$
\end{proof}

\section{Reduction to derivative estimates} \label{sub2}
Proving Theorems \ref{t2} and \ref{t3} can be reduced to a priori estimate on the derivatives of $v$. This is due to the well-known $C^{2,\beta}$ regularity result for fully nonlinear nondegenerate elliptic equations, together with the following lemma.

\begin{lemma}\label{reduction2}
For each $\epsilon>0$, define 
\begin{equation}\label{diffusione}
x_t^{\alpha,x}(\epsilon)=x+\int_0^t\sigma^{\alpha_s}dw_s+\int_0^t\epsilon I d\tilde w_s+\int_0^tb^{\alpha_s}ds,
\end{equation}
where $\tilde w_t$ is a $d$-dimensional Wiener process independent of $w_t$ and $I$ is the identity matrix of size $d\times d$. Let $\tau^{\alpha,x}(\epsilon)$ be the first exit time of $x_t^{\alpha,x}(\epsilon)$ from $D$. Consider the corresponding value function
$$v^\epsilon(x)=\sfa E^\alpha_x\bigg[g\big(x_{\tau(\epsilon)}(\epsilon)\big)e^{-\phi_{\tau(\epsilon)}}+\int_0^{\tau(\epsilon)}f\big(x_t(\epsilon)\big)e^{-\phi_t}dt\bigg].$$
If $f,g\in C^{0,1}(\bar D)$, then we have
\begin{equation}\label{conver}
\lim_{\epsilon\downarrow0}|v^\epsilon-v|_{0,D}=0.
\end{equation}
\end{lemma}
\begin{proof}
Since $f,g$ and $e^{-x}\wedge1$ are all globally Lipschitz, to show (\ref{conver}), it suffices to prove that
\begin{align}
&\lim_{\epsilon\downarrow0}\sup_{x\in D}\sfa E^\alpha_x\sup_{t\le\tau(\epsilon)\wedge\tau}|x_t(\epsilon)-x_t|=0,\label{convx}\\
&\lim_{\epsilon\downarrow0}\sup_{x\in D}\sfa E^\alpha_x|\tau(\epsilon)\vee\tau-\tau(\epsilon)\wedge\tau|=0.\label{convt}
\end{align}

To prove (\ref{convx}), we notice that, for any constant $T\in[1,\infty)$,
\begin{align*}
E^\alpha_x\sup_{t\le\tau(\epsilon)\wedge\tau}|x_t(\epsilon)-x_t|\le&E^\alpha_x\sup_{t\le\tau(\epsilon)\wedge\tau\wedge T}|x_t(\epsilon)-x_t|+KP^\alpha_x(\tau>T)\\
=&\epsilon E^\alpha_x\sup_{t\le\tau(\epsilon)\wedge\tau\wedge T}|\tilde w_t|+\frac{K}{T}E^\alpha_x\tau\\
\le&3\epsilon T+\frac{K}{T}|\psi|_{0,D}.
\end{align*}
By taking the supremum with respect to $\alpha$ on the left side and letting first $\epsilon\downarrow0$ and then $T\uparrow\infty$, we obtain (\ref{convx}).

To prove (\ref{convx}), we notice that
$$E^\alpha_x|\tau(\epsilon)\vee\tau-\tau(\epsilon)\wedge\tau|=E^\alpha_x(\tau-\tau(\epsilon))\mathbbm1_{\tau\ge\tau(\epsilon)}+E^\alpha_x(\tau(\epsilon)-\tau)\mathbbm1_{\tau<\tau(\epsilon)}.$$
Then we estimate both terms. We have
\allowdisplaybreaks\begin{align*}
\MoveEqLeft E^\alpha_x(\tau-\tau(\epsilon))\mathbbm1_{\tau>\tau(\epsilon)}\\
&\le-E_x^\alpha\int_{\tau(\epsilon)\wedge\tau}^{\tau}L\psi(x_t)dt\\
&=E^\alpha_x\Big(\psi (x_{\tau(\epsilon)})-\psi(x_{\tau(\epsilon)}(\epsilon))\Big)\mathbbm1_{\tau(\epsilon)<\tau}\\
&\le E^\alpha_x\Big(\psi (x_{\tau(\epsilon)})-\psi(x_{\tau(\epsilon)}(\epsilon))\Big)\mathbbm1_{\tau(\epsilon)<\tau\le T}+2|\psi|_{0,D}P^\alpha_x(\tau>T)\\
&\le |\psi|_{0,1,D}E^\alpha_x\sup_{t\le \tau(\epsilon)\wedge\tau\wedge T}|x_t-x_t(\epsilon)|+\frac{K}{T}E_x^\alpha\tau.
\end{align*}
Similarly, by notice that for sufficiently small $\epsilon$, 
\begin{equation}\label{Lpsie}
L^\alpha(\epsilon) \psi=L^\alpha\psi+\frac{\epsilon^2}{2}\Delta \psi\le-1/2,
\end{equation}
we have
\begin{align*}
E^\alpha_x(\tau(\epsilon)-\tau)\mathbbm1_{\tau<\tau(\epsilon)}\le&-2E_x^\alpha\int_{\tau(\epsilon)\wedge\tau}^{\tau(\epsilon)}L(\epsilon)\psi(x_t(\epsilon))dt\\
\le&2|\psi|_{0,1,D}E_x^\alpha\sup_{t\le \tau(\epsilon)\wedge\tau\wedge T}|x_t-x_t(\epsilon)|+\frac{K}{T}E_x^\alpha\tau.
\end{align*}
It turns out that
$$E^\alpha_x|\tau(\epsilon)\vee\tau-\tau(\epsilon)\wedge\tau|\le \epsilon KT+\frac{K}{T}.$$
Again, by taking the supremum over $\mathfrak A$ on the left side of the inequality and letting first $\epsilon\downarrow0$ and then $T\uparrow\infty$, we obtain (\ref{convt}).
\end{proof}

Now we state our reduction and explain how it works.

\begin{remark}\label{reducing}
To prove Theorem \ref{t2}, it suffices to establish the first derivative estimate (\ref{e1}) by a priori assuming that $v\in C^{1}(\bar D)$. Similarly, to prove the regularity results in Theorem \ref{t3}, it suffices to establish the second derivative estimate (\ref{e2}) by a priori assuming that $v\in C^{2}(\bar D)$. Moreover, it doesn't hurt to suppose that $f^\alpha,g\in C^{2}(\bar D)$ when estimating the derivatives.
\end{remark}

Indeed, for the controlled diffusion process given by (\ref{diffusione}), its diffusion term is of size $d\times(d+d_1)$ in the form of $\sigma^\alpha(\epsilon)=(\sigma^\alpha|\epsilon I)$. As a result, its associated Bellman equation is
$$\sup_{\alpha\in A}\big[\big(a^\alpha(\epsilon)\big)_{ij} u^\epsilon_{x^ix^j}+(b^\alpha)_i u^\epsilon_{x^i}-c^\alpha u^\epsilon+f^\alpha\big]=0,$$
where $a^\alpha(\epsilon)=a^\alpha+(\epsilon^2/2) I$, which is nondegenerate for each $\epsilon>0$. Suppose that $f^\alpha$ and $g$ are as smooth as we want. By Theorem 7 in Section 6.2 of \cite{MR901759}, we know that for each $\epsilon>0$, the nondegenerate bellman equation with Dirichlet boundary data has a unique solution $u^\epsilon$ in the class of $C^{2,\beta}(\bar D)$. By It\^o's formula and the uniqueness of this PDE problem, we see that $u^\epsilon=v^{\epsilon}$, which implies that the first and second derivatives of $v^\epsilon$ exist up to the boundary, for each $\epsilon>0$. 

Then we may estimate the derivatives of $v^\epsilon$. First, Assupmtion \ref{a1} implies that for sufficiently small $\epsilon>0$,
$$\sup_{\alpha\in A}L^\alpha(\epsilon)\psi(x)=\sup_{\alpha\in A}L^\alpha\psi(x)+\frac{\epsilon^2}{2}\Delta\psi(x)\le-1/2.$$
Second, since
$$qa^\alpha(\epsilon)q^*=qa^{\alpha}q^*+(\epsilon^2/2)I,$$
from Assupmtion \ref{a2}, 
\begin{align*}
&\sup_{q\in\mathbb O^d}\sup_{\alpha\in  A}\big[(qa^\alpha(\epsilon) q^*)_{ij}\gamma^{ij}+(b^\alpha)_ip^i-c^\alpha z+f^\alpha(x)\big]\\
=&\sup_{\alpha\in A}\big[(a^\alpha(\epsilon))_{ij}\gamma^{ij}+(b^\alpha)_ip^i-c^\alpha z+f^\alpha(x)\big].
\end{align*}
They play the same roles as Assuptions \ref{a1} and \ref{a2}, respectively.

Once we obtained the first derivative estimate (\ref{e1}) for $v^\epsilon$, we know that $v^\epsilon$ is locally Lipschitz for each sufficiently small $\epsilon$. Notice that the constant $N$ in (\ref{e1}) doesn't depend on $\epsilon$. Therefore by letting $\epsilon\downarrow0$, we conclude that $v$ is locally Lipschitz, and then obtain the same first derivative estimate a.e. in $D$.

If we have the first inequality in (\ref{e3}) for $v^\epsilon$, which is the second derivative estimate from below, we have
\begin{align*}
\MoveEqLeft v^\epsilon_{(\xi)(\xi)}+N\Big[|x|^2+\psi\Big(\log\frac{\psi}{\kappa}-1\Big)\Big]_{(\xi)(\xi)}\\
&=v^\epsilon_{(\xi)(\xi)}+N\Big(|\xi|^2+\pxsop+\psi_{(\xi)(\xi)}\log\frac{\psi}{\kappa}\Big)\ge0,
\end{align*}
in the set $\{x\in D: \psi\le\kappa\}$. Then we see that the function given in (\ref{e2}) is convex $\{x\in D: \psi\le\kappa\}$. Again, the constant $N$ here doesn't on $\epsilon$. By letting $\epsilon\downarrow0$ we have the same conclusion for $v$. If we furthermore have the second inequality in (\ref{e3}) for $v^\epsilon$, then we know that the derivatives of $v^{\epsilon}$ are locally Lipschitz, by letting $\epsilon\downarrow0$, we conclude that the second derivatives of $v$ exist almost everywhere, and satisfy the second derivative estimate (\ref{e3}).

Observe further that for each fixed $\epsilon>0$, the functions $f^\alpha$ and $g$ can be uniformly approximated in $\bar D$ by sufficiently smooth functions, in such a way that the constant $N$ in (\ref{e1}) and (\ref{e3}) increases by at most a factor of two when $f^\alpha$ and $g$ are replaced with the approximating functions. Therefore we may suppose $f,g\in C^2(\bar D)$ when estimating the derivatives.

\section{Quasiderivatives and auxiliary convergence results}\label{sub3}

In this section, we introduce the quasiderivatives and collect auxiliary convergence results to be used repetitively Sections \ref{sub5} and \ref{sub6}.

For each $\alpha\in\mathfrak A$, let $r_t^\alpha$, $\hat r_t^\alpha$, $\pi_t^\alpha$, $\hat\pi_t^\alpha$, $P_t^\alpha$, $\hat P_t^\alpha$ be jointly measurable adapted processes with values in $\mathbb R$, $\mathbb R$, $\mathbb R^{d_1}$, $\mathbb R^{d_1}$, $\operatorname{Skew}(d,\mathbb R)$, $\operatorname{Skew}(d,\mathbb R)$, respectively, where $\operatorname{Skew}(d,\mathbb R)$ denotes the set of all $d\times d$-size skew-symmetric matrices. Let $\epsilon\in[0,1]$ be constant. For each $\alpha\in\mathfrak A$, $x,y,z\in  D$, $\xi,\eta\in \Rd$, we consider the diffusion process defined by (\ref{itox}) and the following six other stochastic processes:
\allowdisplaybreaks\begin{align}
\label{itoy}
dy_t^{\alpha,y}( \epsilon)=&\sqrt{\theta_t^\alpha(\epsilon)}e^{ \epsilon P_t^{\alpha}}\sigma^{\alpha_t}dw_t+\Big[\theta_t^\alpha(\epsilon) b^{\alpha_t}-\sqrt{\theta_t^\alpha(\epsilon)}e^{ \epsilon P_t^{\alpha}}\sigma^{\alpha_t} \epsilon\pi_t^{\alpha}\Big]dt,\\
\label{itoz}
dz_t^{\alpha,z}(\epsilon)=&\sqrt{\hat\theta_t^\alpha(\epsilon)}e^{ \epsilon P_t^{\alpha}}e^{\frac{\epsilon^2}{2}  \hat P_t^{\alpha}}\sigma^{\alpha_t}dw_t\\
\nonumber&+\Big[\hat\theta_t^\alpha(\epsilon) b^{\alpha_t}-\sqrt{\hat\theta_t^\alpha(\epsilon)}e^{ \epsilon P_t^{\alpha}}e^{\frac{\epsilon^2}{2} \hat P_t^{\alpha}}\sigma^{\alpha_t}( \epsilon\pi_t^{\alpha}+\frac{\epsilon^2}{2} \hat\pi_t^{\alpha})\Big]dt,\\
\label{itoxi}
d\xi_t^{\alpha,\xi}=&\Big[r^\alpha_t\sigma^{\alpha_t}+P^\alpha_t\sigma^{\alpha_t} \Big]dw_t+\Big[2r^\alpha_tb^{\alpha_t}-\sigma^{\alpha_t} \pi^\alpha_t\Big]dt,\\
\label{itoeta}
d\eta_t^{\alpha,\eta}=&\Big[\big(\hat{r}^\alpha_t-(r^\alpha_t)^2\big)\sigma^{\alpha_t}+\big(\hat{P}^\alpha_t+(P^\alpha_t)^2+2r^\alpha_tP^\alpha_t \big)\sigma^{\alpha_t} \Big]dw_t\\
\nonumber &+\Big[2\hat{r}^\alpha_t b^{\alpha_t}-\sigma^{\alpha_t}\hat{\pi}^\alpha_t-2\big(r^\alpha_t\sigma^{\alpha_t}+P^\alpha_t\sigma^{\alpha_t} \big) \pi^\alpha_t\Big]dt,\\
\label{itoxia}
d\tilde\xi_t^{\alpha,0}=&\pi^\alpha_tdw_t,\\
\label{itoetaa}
d\tilde\eta_t^{\alpha,0}=&\hat\pi^\alpha_tdw_t+d\big(\tilde\xi_t^{\alpha,0}\big)^2-d\big\langle\tilde\xi^{\alpha,0}\big\rangle_t,
\end{align}
where
\begin{align}
\label{theta}
\theta_t^\alpha(\epsilon)=&1+\frac{1}{\pi}\arctan\big(\pi2\epsilon r^\alpha_t\big),\\
 \hat\theta_t^\alpha(\epsilon)=&1+\frac{1}{\pi}\arctan\Big[\pi\big(2\epsilon r_t+\epsilon^2\hat r_t^\alpha\big)\Big].\label{thetahat}
\end{align}

In (\ref{itoy}) and (\ref{itoz}), notice that when $\epsilon=0$, we have $x_t^{\alpha,y}$ and $x_t^{\alpha,z}$. In Lemmas \ref{thm1} and \ref{thm2}, we will prove that under suitable conditions, $\xi_t^{\alpha,\xi}$ and $\eta_t^{\alpha,\eta}$, given by (\ref{itoxi}) and (\ref{itoeta}) respectively,  are the first derivative of $y_t^{\alpha,x+\epsilon\xi}(\epsilon)$ and the second derivative of $z_t^{\alpha, x+\epsilon\xi+\epsilon^2\eta/2}(\epsilon)$ in an appropriate sense (see (\ref{res1b2}) and (\ref{res2b3})), respectively. We call $\xi_t^{\alpha,\xi}$ the first quasiderivative, $\eta_t^{\alpha,\eta}$ the second quaisiderivative, $\tilde\xi_t^{\alpha,0}$ the first adjoint quasiderivative and $\tilde\eta_t^{\alpha,0}$ the second adjoint quasiderivative. The auxiliary processes $r_t^\alpha$ and $\hat r_t^\alpha$ come from random time change, $\pi_t^\alpha$ and $\hat\pi_t^\alpha$ are due to Girsanov's theorem on changing the probability space, and $P_t^\alpha$ and $\hat P_t^\alpha$ appear in order to utilize Assumption \ref{a2}. 

Sufficient conditions should be given on the auxiliary processes such that (\ref{itoy})-(\ref{itoetaa}) are meaningful. Note that, in the next section, we will define the auxiliary processes $r_t^\alpha$, $\pi_t^\alpha$, and $P_t^\alpha$ as functions of $\xi_t^{\alpha,\xi}$, therefore (\ref{itoxi}) will be a stochastic differential equation. We provide the following lemma which is applicable to the quasiderivatives defined in next section.

\begin{lemma}\label{meaningful}
For each $\alpha\in \mathfrak A$, if  $r^\alpha_t$, $(\pi_t^\alpha)^k$ and $(P_t^{\alpha})^{ij}$ are all in the form of $(\xi_t^{\alpha,\xi},\mathfrak p_t^\alpha)$,
where $\mathfrak p_t^\alpha$ is independent of $\xi_t^{\alpha,\xi}$ and satisfies $|\mathfrak p_t^\alpha|\le C$, and $\hat r_t^\alpha$, $(\hat\pi_t^\alpha)^k$ and $(\hat P_t^\alpha)^{ij}$ are all in the form of $\mathfrak q_t^\alpha$, which is independent of $\eta_t^{\alpha,\eta}$ and satisfies $\int_0^t|\mathfrak q_t^\alpha|^2<\infty$ a.s., for all $t\ge0$, then (\ref{itoxi}) has a unique solution up to indistinguishability, and 
(\ref{itoy})-(\ref{itoetaa}) are well-defined.
\end{lemma}

\begin{proof} Since $r^\alpha_t$, $(\pi_t^\alpha)^k$ and $(P_t^{\alpha})^{ij}$ are affine functions of $\xi_t^{\alpha,\xi}$, by applying Theorem V.1.1 in \cite{MR1311478}, we conclude that (\ref{itoxi}) has a unique solution up to indistinguishability. Moreover, apply Lemma 3.1(1) in \cite{MR3047001} to $\xi_t^{\alpha,\xi}$ with $M_t^\alpha=0$, we have, for any constants $T, p\in(0,\infty)$,
\begin{equation}\label{xip}
\sup_{\alpha\in\mathfrak A}E^\alpha_\xi\sup_{t\le\gamma\wedge T}|\xi_t|^p\le\left\{
\begin{array}{ll}
\displaystyle e^{NT}|\xi|^p&\mbox{ if }p\ge2\\
\displaystyle e^{NT}\frac{4-p}{2-p}|\xi|^p&\mbox{ if }p<2,
\end{array}
\right. 
\end{equation}
which implies that
$$\int_0^T|\xi_t^{\alpha,\xi}|^p<\infty, \mbox{ a.s.. }$$

To prove that (\ref{itoy}) is well-defined, it suffices to show that for any $T\in(0,\infty)$, a.s.
\allowdisplaybreaks\begin{equation*}
\begin{gathered}
\int_0^T\Big\|\sqrt{\theta_t^\alpha(\epsilon)}e^{ \epsilon P_t^{\alpha}}\sigma^{\alpha_t}\Big\|^2+\Big|\theta_t^\alpha(\epsilon)b^{\alpha_t}-\sqrt{\theta_t^\alpha(\epsilon)}e^{ \epsilon P_t^{\alpha}}\sigma^{\alpha_t} \epsilon\pi_t^{\alpha}\Big|dt<\infty,
\end{gathered}
\end{equation*}
which is true since
\begin{equation*}
\begin{gathered}
\|\sqrt{\theta_t^\alpha(\epsilon)}e^{ \epsilon P_t^{\alpha}}\sigma^{\alpha_t}\|\le \sqrt{3/2}K_0,
\qquad|\theta_t^\alpha(\epsilon)b^{\alpha_t}|\le(3/2)K_0,\\
|\sqrt{\theta_t^\alpha(\epsilon)}e^{ \epsilon P_t^{\alpha}}\sigma^{\alpha_t} \epsilon\pi_t^{\alpha}|\le\sqrt{3/2}K_0C|\xi_t^{\alpha,\xi}|.
\end{gathered}
\end{equation*}

Similarly, we can prove that (\ref{itoz}), (\ref{itoeta})-(\ref{itoetaa}) are well-defined.
\end{proof}

In the next two lemmas, we collect convergence results to be used in the proof of the main theorems.

Let $U$ be a connected open subset of $D$. Define
\begin{equation*}
\begin{gathered}
\tau_U^{\alpha,x}=\inf\{t\ge0:x_t^{\alpha,x}\notin U\},\\
\bar\tau_U^{\alpha,y}(\epsilon)=\inf\{t\ge0:y_t^{\alpha,y}(\epsilon)\notin U\},\qquad
\hat\tau_U^{\alpha,z}(\epsilon)=\inf\{t\ge0:z_t^{\alpha,z}(\epsilon)\notin U\}.
\end{gathered}
\end{equation*}

\begin{lemma}\label{thm1} Suppose that the assumptions on $r_t^\alpha$, $\hat r_t^\alpha$, $\pi_t^\alpha$, $\hat\pi_t^\alpha$, $P_t^\alpha$ and $\hat P_t^\alpha$ in Lemma \ref{meaningful} hold. Given any $x\in U$, $\xi\in\Rd$ and constants $p\in (0,\infty)$, $p'\in[0,p)$, $T\in [1,\infty)$, we have the following results. 

Given stopping times $\gamma^\alpha$ satisfying $\gamma^\alpha\le\tau_U^{\alpha,x}$ for each $\alpha\in\mathfrak A$, we have
\begin{equation}\label{res1a}
\sup_{\alpha\in\mathfrak A}E^\alpha_\xi\sup_{t\le\gamma\wedge T}|\xi_t|^p<\infty.
\end{equation}

Let the constant $\epsilon_0\le1$ be sufficiently small so that $B(x,\epsilon_0|\xi|)\subset U$. For any $\epsilon\in[0,\epsilon_0]$, given stopping times $\gamma^\alpha(\epsilon)$ such that 
$$\gamma^\alpha(\epsilon)\le\tau_U^{\alpha,x}\wedge\bar\tau_U^{\alpha,x+\epsilon\xi}(\epsilon)$$ for each $\alpha\in\mathfrak A$, we have
\begin{equation}\label{res1b1}
\lim_{\epsilon\downarrow0}\sup_{\alpha\in\mathfrak A}E\sup_{t\le\gamma^\alpha(\epsilon)\wedge T}\frac{|y_t^{\alpha,x+\epsilon\xi}(\epsilon)-x_t^{\alpha,x}|^p}{\epsilon^{p'}}=0,
\end{equation}
\begin{equation}\label{res1b2}
\lim_{\epsilon\downarrow0}\sup_{\alpha\in\mathfrak A}E\sup_{t\le\gamma^\alpha(\epsilon)\wedge T}\Big|\frac{y_t^{\alpha,x+\epsilon\xi}(\epsilon)-x_t^{\alpha,x}}{\epsilon}-\xi_t^{\alpha,\xi}\Big|^p=0.
\end{equation}
If for each $\alpha\in A$, the function $h^\alpha: \bar U\rightarrow \mathbb R$ is in the class of $C^{0,1}(\bar U)$, and the Lipschitz constants of $h^\alpha$ are uniformly bounded with respect to $\alpha$, then we have
\begin{equation}\label{hres1b1}
\lim_{\epsilon\downarrow0}\sup_{\alpha\in\mathfrak A}E\sup_{t\le\gamma^\alpha(\epsilon)\wedge T}\frac{|h^{\alpha_t}(y_t^{\alpha,x+\epsilon\xi}(\epsilon))-h^{\alpha_t}(x_t^{\alpha,x})|^p}{\epsilon^{p'}}=0.
\end{equation}
If furthermore $h^\alpha\in C^1(\bar U)$, and $h_x^\alpha$ are  continuous in $x$, uniformly with respect to $\alpha$, then we have
\begin{equation}\label{hres1b2}
\lim_{\epsilon\downarrow0}\sup_{\alpha\in\mathfrak A}E\sup_{t\le\gamma^\alpha(\epsilon)\wedge T}\bigg|\frac{h^{\alpha_t}(y_t^{\alpha,x+\epsilon\xi}(\epsilon))-h^{\alpha_t}(x_t^{\alpha,x})}{\epsilon}-h^{\alpha_t}_{(\xi_t^{\alpha,\xi})}(x_t^{\alpha,x})\bigg|^p=0.
\end{equation}
\end{lemma}
\begin{proof}
In the proof, we drop the superscripts $\alpha$, $\alpha_t$, etc., when this will not cause confusion.

The first property (\ref{res1a}) has been proved in Lemma \ref{meaningful}, see (\ref{xip}).

To prove the others we first consider the It\^o stochastic equations (3.1) and (3.2) in \cite{MR3047001} where
$$\zeta_t^{\alpha,\zeta}=x_t^{\alpha,x},\qquad\zeta_t^{\alpha,\zeta(\epsilon)}(\epsilon)=y_t^{\alpha,x+\epsilon\xi}(\epsilon).$$
Notice that
\begin{align*}
|\sqrt{\theta_t(\epsilon)}-1|=&\frac{(1/\pi)|\arctan(\pi2\epsilon r_t)|}{\sqrt{\theta_t(\epsilon)}+1}\le 2\epsilon |r_t|,\\
\|e^{\epsilon P_t}-I_{d\times d}\|=&\epsilon \|P_t e^{\epsilon' P_t}\|,
\end{align*}
where $\epsilon'\in(0,\epsilon)$ is non-constant and due to Mean Value Theorem. Therefore,
\begin{align*}
\|\kappa_t(y,\epsilon)-\kappa_t(x)\|\le&\|\sigma\|\big[(\sqrt{\theta_t(\epsilon)}-1)\|e^{\epsilon P_t}\|+\|e^{\epsilon P_t}-I_{d\times d}\|\big]\\
\le&\epsilon K_0 C|\xi_t|3\sqrt d,\\
|\nu_t(y,\epsilon)-\nu_t(x)|\le&2\epsilon|r_t|\|\sigma\|+(3/2)\|\sigma\|\sqrt d\epsilon|\pi_t|\\
\le&\epsilon K_0\big[2+(3/2)\sqrt d \big]C|\xi_t|.
\end{align*}
Applying Lemma 3.1(2) in \cite{MR3047001} with $M=0$ and $M_t^\alpha=M(K_0,C,d)|\xi_t^{\alpha,\xi}|$, we have
$$\sup_{\alpha\in\mathfrak A}E\sup_{t\le\gamma^\alpha(\epsilon)\wedge T}|y_t^{\alpha,x+\epsilon\xi}(\epsilon)-x_t^{\alpha,x}|^p$$
\begin{equation*}
\le\left\{
\begin{array}{ll}
\displaystyle\epsilon^pe^{NT}\bigg[|\xi|^p+(2p-1)\sup_{\alpha\in\mathfrak A}E^\alpha\int_0^{\gamma(\epsilon)\wedge T}M_t^pdt\bigg]&\mbox{if }p\ge2\\
\displaystyle\epsilon^pe^{NT}\frac{4-p}{2-p}\bigg[|\xi|^p+3^{p/2}\Big(\sup_{\alpha\in\mathfrak A}E^\alpha\int_0^{\gamma(\epsilon)\wedge T}M_t^2dt\Big)^{p/2}\bigg]&\mbox{if }p<2.
\end{array}
\right. 
\end{equation*}
Due to (\ref{res1a}), we have
$$\sup_{[0,\epsilon_0]}\sup_{\alpha\in\mathfrak A}E^\alpha\int_0^{\gamma(\epsilon)\wedge T}M_t^{2\vee p}dt<\infty,$$
which completes the proof of (\ref{res1b1}).

We next consider the It\^o stochastic equations (3.1) and (3.2) in \cite{MR3047001} with 
$$\zeta_t^{\alpha,\zeta}=\xi_t^{\alpha,\xi},\qquad\zeta_t^{\alpha,\zeta(\epsilon)}(\epsilon)=\xi_t^{\alpha,\xi}(\epsilon):=\frac{y_t^{\alpha,x+\epsilon\xi}(\epsilon)-x_t^{\alpha,x}}{\epsilon}.$$
Observe that
\begin{align}
\bigg|\frac{\theta_t(\epsilon)-1}{\epsilon}-2r_t\bigg|=&\frac{\epsilon}{2}\theta''_t(\epsilon')\le\epsilon 4\pi r_t^2,\label{guji}\\
\nonumber\bigg|\frac{\sqrt{\theta_t(\epsilon)}-1}{\epsilon}-r_t\bigg|=&\frac{\epsilon}{2}\bigg|\frac{\theta_t''(\epsilon')\theta_t(\epsilon')-(\theta'_t(\epsilon'))^2/2}{2(\theta_t(\epsilon'))^{3/2}}\bigg|\le\epsilon Cr_t^2,\\
\nonumber\bigg\|\frac{e^{\epsilon P_t}-1}{\epsilon}-P_t\bigg\|=&\frac{\epsilon}{2}\big\|P_t^2e^{\epsilon'P_t}\big\|\le\frac{\epsilon}{2}\|P_t^2\|.
\end{align}
The equation (\ref{res1b2}) can be proved by mimicking the proof of (\ref{res1b1}).

To prove (\ref{hres1b1}), it suffices to notice that
$$|h^\alpha(y)-h^\alpha(x)|\le\sup_{\alpha\in A}|h^\alpha|_{0,1,U}|y-x|,$$
and then apply (\ref{res1b1}).

To prove (\ref{hres1b2}), we notice that
\begin{align*}
\MoveEqLeft\Big|\frac{h^\alpha(y)-h^\alpha(x)}{\epsilon}-h^\alpha_{(\xi)}(x)\Big|\\
\le&\Big|\Big(h^\alpha_x(\lambda x+(1-\lambda)y)-h^\alpha_x(x),\frac{y-x}{\epsilon}\Big)\Big|+\Big|\Big(h^\alpha_x(x),\frac{y-x}{\epsilon}-\xi\Big)\Big|\\
\le&\Big|\Big(h^\alpha_x(\lambda x+(1-\lambda)y)-h^\alpha_x(x),\frac{y-x}{\epsilon}\Big)\Big|(\mathbbm 1_{|y-x|<\delta_n}+\mathbbm 1_{|y-x|\ge\delta_n})\\
&+\sup_{\alpha\in A}|h^\alpha|_{0,1,U}\Big|\frac{y-x}{\epsilon}-\xi\Big|,
\end{align*}
where for each $n\in \mathbb N$, $\delta_n=\delta_n(x)$ is a positive such that
$$\sup_{\alpha\in A}|h^\alpha_x(y)-h^\alpha_x(x)|\le 1/n, \ \forall |y-x|<\delta_n.$$
It follows that
\begin{align*}
\MoveEqLeft\Big|\frac{h^\alpha(y)-h^\alpha(x)}{\epsilon}-h^\alpha_{(\xi)}(x)\Big|\\
\le&\frac{1}{n}\frac{|y-x|}{\epsilon}+\sup_{\alpha\in A}|h^\alpha|_{0,1,U}\bigg(2\frac{|y-x|}{\epsilon}\mathbbm 1_{|y-x|>\delta_n}+\Big|\frac{y-x}{\epsilon}-\xi\Big|\bigg).
\end{align*}
Therefore, for $p\ge1$,
\begin{equation*}
\begin{gathered}
\sup_{\alpha\in\mathfrak A}E\sup_{t\le\gamma^\alpha(\epsilon)\wedge T}\bigg|\frac{h^{\alpha_t}(y_t^{\alpha,x+\epsilon\xi}(\epsilon))-h^{\alpha_t}(x_t^{\alpha,x})}{\epsilon}-h^{\alpha_t}_{(\xi_t^{\alpha,\xi})}(x_t^{\alpha,x})\bigg|^p\\
\le 3^pI_1(\epsilon,n)+3^p\sup_{\alpha\in A}|h^\alpha|^p_{0,1,U}\big(2^pI_2(\epsilon,n)+I_1(\epsilon)\big),
\end{gathered}
\end{equation*}
where
\begin{align*}
I_1(\epsilon,n)=&\frac{1}{n}\sup_{\alpha\in\mathfrak A}E\sup_{t\le\gamma^\alpha(\epsilon)\wedge T}\frac{|y_t^{\alpha,x+\epsilon\xi}-x_t^{\alpha,x}|^p}{\epsilon^p},\\
I_2(\epsilon,n)=&\sup_{\alpha\in\mathfrak A}E\bigg(\sup_{t\le\gamma^\alpha(\epsilon)\wedge T}\frac{|y_t^{\alpha,x+\epsilon\xi}-x_t^{\alpha,x}|^p}{\epsilon^p}\mathbbm1_{|y_t^{\alpha,x+\epsilon\xi}-x_t^{\alpha,x}|>\delta_n}\bigg)\\
\le&\frac{1}{\delta_n^p}\sup_{\alpha\in\mathfrak A}E\sup_{t\le\gamma^\alpha(\epsilon)\wedge T}\frac{|y_t^{\alpha,x+\epsilon\xi}-x_t^{\alpha,x}|^{2p}}{\epsilon^p},\\
I_3(\epsilon)=&\sup_{\alpha\in\mathfrak A}E\sup_{t\le\gamma^\alpha(\epsilon)\wedge T}\Big|\frac{y_t^{\alpha,x+\epsilon\xi}-x_t^{\alpha,x}}{\epsilon}-\xi_t^{\alpha,\xi}\Big|^p.
\end{align*}
By first letting $\epsilon\downarrow0$ and then $n\uparrow\infty$, (\ref{hres1b2}) is verified.
\end{proof}

The next lemma is the second order counterpart of the previous lemma.
\begin{lemma}\label{thm2}
Suppose that the assumptions on $r_t^\alpha$, $\hat r_t^\alpha$, $\pi_t^\alpha$, $\hat\pi_t^\alpha$, $P_t^\alpha$ and $\hat P_t^\alpha$ in Lemma \ref{meaningful} hold. Given any $x\in U$, $\xi,\eta\in \Rd$ and constants $p\in (0,\infty)$, $p'\in[0,p)$, $T\in [1,\infty)$, $x\in D$, $\xi\in\Rd$, $\eta\in\Rd$. We have the following results.

Given stopping times $\gamma^\alpha$ satisfying $\gamma^\alpha\le\tau_U^{\alpha,x}$ for each $\alpha\in\mathfrak A$, we have (\ref{res1a}) and
\begin{equation}\label{res2a}
\sup_{\alpha\in\mathfrak A}E^\alpha_\eta\sup_{t\le\gamma\wedge T}|\eta_t|^p<\infty.
\end{equation}

Let the constant $\epsilon_0\le1$ be sufficiently small so that $B(x,\epsilon_0|\xi|+\epsilon_0^2|\eta|/2)\subset U$. For any $\epsilon\in[-\epsilon_0,\epsilon_0]$, given stopping times $\gamma^\alpha(\epsilon)$ satisfying
$$\gamma^\alpha(\epsilon)\le\tau_U^{\alpha,x}\wedge\hat\tau_U^{\alpha,x+\epsilon\xi+\epsilon^2\eta/2}(\epsilon)\wedge\hat\tau_U^{\alpha,x-\epsilon\xi+\epsilon^2\eta/2}(-\epsilon)$$ 
for each $\alpha\in\mathfrak A$, we have\begin{equation}\label{res2b1}
\lim_{\epsilon\rightarrow0}\sup_{\alpha\in\mathfrak A}E\sup_{t\le\gamma^\alpha(\epsilon)\wedge T}\frac{|z_t^{\alpha,x+\epsilon\xi+\epsilon^2\eta/2}(\epsilon)-x_t^{\alpha,x}|^p}{\epsilon^{p'}}=0,
\end{equation}
\begin{equation}\label{res2b2}
\lim_{\epsilon\rightarrow0}\sup_{\alpha\in\mathfrak A}E\sup_{t\le\gamma^\alpha(\epsilon)\wedge T}\Big|\frac{z_t^{\alpha,x+\epsilon\xi+\epsilon^2\eta/2}(\epsilon)-x_t^{\alpha,x}}{\epsilon}-\xi_t^{\alpha,\xi}\Big|^p=0,
\end{equation}
\begin{equation}\label{res2b3}
\lim_{\epsilon\rightarrow0}\sup_{\alpha\in\mathfrak A}E\sup_{t\le\gamma^\alpha(\epsilon)\wedge T}\bigg|\frac{z_t^{\alpha,x+\epsilon\xi+\epsilon^2\eta/2}(\epsilon)-2x_t^{\alpha,x}+z_t^{\alpha,x-\epsilon\xi+\epsilon^2\eta/2}(-\epsilon)}{\epsilon^2}-\eta_t^{\alpha,\eta}\bigg|^p=0.
\end{equation}
If for each $\alpha\in A$, the function $h^\alpha: \bar U\rightarrow\mathbb R$ is in the class of $C^{0,1}(\bar U)$, and the Lipschitz constants of $h^\alpha$ are uniformly bounded in $\alpha$, then we have
\begin{equation}\label{hres2b1}
\lim_{\epsilon\rightarrow0}\sup_{\alpha\in\mathfrak A}E\sup_{t\le\gamma^\alpha(\epsilon)\wedge T}\frac{|h^{\alpha_t}(z_t^{\alpha,x+\epsilon\xi+\epsilon^2\eta/2}(\epsilon))-h^{\alpha_t}(x_t^{\alpha,x})|^p}{\epsilon^{p'}}=0.
\end{equation}
If furthermore $h^\alpha\in C^1(\bar U)$, and $h_x^\alpha$ are uniformly continuous in $\alpha$, then we have
\begin{equation}\label{hres2b2}
\lim_{\epsilon\rightarrow0}\sup_{\alpha\in\mathfrak A}E\sup_{t\le\gamma^\alpha(\epsilon)\wedge T}\bigg|\frac{h^{\alpha_t}(z_t^{\alpha,x+\epsilon\xi+\ve^2\eta/2}(\epsilon))-h^{\alpha_t}(x_t^{\alpha,x})}{\epsilon}-h^{\alpha_t}_{(\xi_t^{\alpha,\xi})}(x_t^{\alpha,x})\bigg|^p=0.
\end{equation}
If furthermore $h^\alpha\in C^2(\bar U)$, and $h_{xx}^\alpha$ are uniformly continuous in $\alpha$, then we have
\begin{equation}\label{hres2b3}
\begin{gathered}
\lim_{\epsilon\rightarrow0}\sup_{\alpha\in\mathfrak A}E\sup_{t\le\gamma^\alpha(\epsilon)\wedge T}\bigg|\frac{h^{\alpha_t}(z_t^{\alpha,x+\epsilon\xi+\ve^2\eta/2}(\epsilon))-2h^{\alpha_t}(x_t^{\alpha,x})+h^{\alpha_t}(z_t^{\alpha,x-\epsilon\xi+\ve^2\eta/2}(-\epsilon))}{\epsilon^2}\\
-h^{\alpha_t}_{(\xi_t^{\alpha,\xi})(\xi_t^{\alpha,\xi})}(x_t^{\alpha,x})-h^{\alpha_t}_{(\eta_t^{\alpha,\eta})}(x_t^{\alpha,x})\bigg|^p=0.
\end{gathered}
\end{equation}
\end{lemma}

\begin{proof}
Again, we may drop superscripts $\alpha$, $\alpha_t$, etc., when this will cause no confusion.

The trueness of inequality (\ref{res2a}) is obvious due to the assumptions on $r_t^\alpha$, $\hat r_t^\alpha$, $\pi_t^\alpha$, $\hat\pi_t^\alpha$, $P_t^\alpha$ and $\hat P_t^\alpha$ given in Lemma \ref{meaningful} and the inequality (\ref{res1a}).

The equations (\ref{res2b1}) and (\ref{res2b2}) can be obtained by repeating the proof of (\ref{res1b1}) and (\ref{res1b2}).

To proof (\ref{res2b3}), we consider the It\^o stochastic equations (3.1) and (3.2) in \cite{MR3047001} with
\begin{align*}
\zeta_t^{\alpha,\zeta}=\eta_t^{\alpha,\eta},\qquad \zeta_t^{\alpha,\zeta(\epsilon)}(\epsilon)=\frac{z_t^{\alpha,x+\epsilon\xi+\epsilon^2\eta/2}(\epsilon)-2x_t^{\alpha,x}+z_t^{\alpha,x-\epsilon\xi+\epsilon^2\eta/2}(-\epsilon)}{\epsilon^2},
\end{align*}
and then mimic the proof of (\ref{res1b2}).

Finally, (\ref{hres2b1})-(\ref{hres2b3}) are nothing but staightforward extensions of (\ref{hres1b1}) and (\ref{hres1b2}).
\end{proof}

We end up this section by showing a convergence result about the stopping times to be applied in the proof of the main theorems.

\begin{lemma}\label{markovtime}
Suppose that Assumption \ref{a1} holds. 
\begin{enumerate}
\item Let $T$ be deterministic time. We have
\begin{equation}\label{tauT}
\lim_{T\uparrow\infty}\sup_{\alpha\in\mathfrak A}E^\alpha_x\big(\tau_D-\tau_D\wedge T\big)=0.
\end{equation}
\item
If (\ref{res1b1}) holds with $p=1$, $p'=0$, $\gamma^\alpha(\epsilon)=\tau_D^{\alpha,x}\wedge\bar \tau_D^{\alpha,x+\epsilon\xi}(\epsilon)$,
for all $T\in[1,\infty)$, then we have
\begin{equation}\label{stopping1}
\lim_{\epsilon\downarrow0}\sup_{\alpha\in\mathfrak A}E\big(\tau^{\alpha,x}_D-\gamma^\alpha(\epsilon)\big)=0.
\end{equation}
\item
If (\ref{res2b1}) hold with $p=1$, $p'=0$, $\gamma^\alpha(\epsilon)=\tau_D^{\alpha,x}\wedge\hat \tau_D^{\alpha,x+\epsilon\xi+\epsilon^2\eta/2}(\epsilon)\wedge \hat \tau_D^{\alpha,x-\epsilon\xi+\epsilon^2\eta/2}(-\epsilon)$,
for all $T\in [1,\infty)$, then we have
\begin{equation}\label{stopping2}
\lim_{\epsilon\downarrow0}\sup_{\alpha\in\mathfrak A}E\big(\tau^{\alpha,x}_D-\gamma^\alpha(\epsilon)\big)=0.
\end{equation}
\end{enumerate}
All of the statements above are still true when replacing $D$ with $D_\delta=\{x\in D: \psi>\delta\}.$
\end{lemma}
\begin{proof}
We drop the subscript $D$ and the argument $\epsilon$ for simplicity of the notation. 

We first observe that, for each $\alpha\in\mathfrak A$,
\begin{align*}
E^\alpha_x(\tau-\tau\wedge T)\le&-E^\alpha_x\int_{\tau\wedge T}^{\tau}L^\alpha\psi(x_t) dt=E_x^\alpha\psi(x_{\tau\wedge T})\mathbbm1_{\tau>T}\le|\psi|_{0,D}\frac{E^\alpha_x\tau}{T},
\end{align*}
which implies (\ref{tauT}). 


Next, notice that, for any $\alpha\in\mathfrak A$,
\allowdisplaybreaks\begin{align*}
E\big(\tau^{\alpha,x}-\gamma^\alpha\big)\le&-E\int_{\gamma^\alpha}^{\tau^{\alpha,x}}L^\alpha\psi(x_t^{\alpha,x})dt\\
=&E\Big(\psi\big(x_{\bar\tau^{\alpha,x+ \epsilon\xi}}^{\alpha,x}\big)-\psi\big(y_{\bar\tau^{\alpha,x+ \epsilon\xi}}^{\alpha,x+ \epsilon\xi}\big)\Big)
\mathbbm1_{\bar\tau^{\alpha,x+ \epsilon\xi}<\tau^{\alpha,x}}\\
\le&E\Big(\psi\big(x_{\bar\tau^{\alpha,x+ \epsilon\xi}}^{\alpha,x}\big)-\psi\big(y_{\bar\tau^{\alpha,x+ \epsilon\xi}}^{\alpha,x+ \epsilon\xi}\big)\Big)
\mathbbm1_{\bar\tau^{\alpha,x+ \epsilon\xi}<\tau^{\alpha,x}\le T}+2|\psi|_{0,D}\frac{E^\alpha_x\tau}{T}.
\end{align*}
Due to (\ref{hres1b1}), we have
\begin{align*}
\lim_{T\uparrow0}\lim_{\epsilon\downarrow0}\sup_{\alpha\in\mathfrak A}E(\tau^{\alpha,x}_D-\tau^{\alpha,x}_D\wedge\bar\tau^{\alpha,x+\epsilon\xi}_D(\epsilon))=0.
\end{align*}

To prove (\ref{stopping2}), we just need to notice that for any stopping times $\tau$, $\tau_1$ and $\tau_2$, we have
$$\tau-\tau\wedge\tau_1\wedge\tau_2=(\tau-\tau\wedge\tau_1)I_{\tau_1<\tau_2}+(\tau-\tau\wedge\tau_2)I_{\tau_1\ge\tau_2}.$$

For the conclusions when the domain is $D_\delta$, it suffices to repeat the proof with $\psi$ replaced with $\psi-\delta$.


\end{proof}

\section{Construction of barriers and quasiderivatives}\label{sub4}

For constants $\delta$ and $\lambda$ satisfying $0<\delta<\lambda$, define
\begin{equation*}
\begin{gathered}
D_\delta=\{x\in D:\delta<\psi(x)\},\qquad D^\lambda=\{x\in D: \psi(x)<\lambda\},\\
D_\delta^\lambda=\{x\in D: \delta<\psi(x)<\lambda\}.
\end{gathered}
\end{equation*}
Here we construct two barriers. The boundary barrier $\B_1(x,\xi)$ is defined on $D_\delta^\lambda\times\Rd$, and the interior barrier $\B_2(x,\xi)$ is defined on  $D_{\lambda^2}\times\Rd$, where $\lambda\in(0,1)$ is a sufficiently small constant throughout this article which will be determined in the proof of Lemmas \ref{l1}, \ref{l2} and \ref{l3}, and $\delta$ is an arbitrary constant in the interval $(0,\lambda^2)$ in this section, which will approach zero in the proof of Theorems \ref{t2} and \ref{t3}. We construct the barriers and quasiderivatives in such a way that $\B_1(x_t^{\alpha,x},\xi_t^{\alpha,\xi})$ and $\B_2(x_t^{\alpha,x},\xi_t^{\alpha,\xi})$ are local supermartingales.

Due to Assumption \ref{a1}, we suppose that
$$4(\|\sigma\|^2_{0,A}+|b|^2_{0,A})\le-\sup_{\alpha\in A}L^\alpha\psi(x),\ \forall x\in D,$$
by replacing $\psi(x)$ with $4(\|\sigma\|^2_{0,A}+|b|^2_{0,A})\psi(x)$.

\begin{lemma}\label{l1}

In $D_\delta^\lambda\times\Rd$, let
$$
\mathrm{B}_1(x,\xi)=\gamma\bigg[\beta|\xi|^2+\pxsop\bigg],
$$
with
$$\beta=1+\frac{1}{8\lambda}\psi\Big(1-\frac{1}{4\lambda}\psi\Big),\qquad\gamma=\lambda^2+\psi\Big(1-\frac{1}{4\lambda}\psi\Big).$$

For each $\alpha$, we define the first and second quasiderivatives by (\ref{itoxi}) and (\ref{itoeta}), in which 
\begin{itemize}
\item $r_t^\alpha:=r(x_t^{\alpha,x},\xi_t^{\alpha,\xi})$, where $\displaystyle r(x,\xi):=\rho(x,\xi)+\frac{\psi_{(\xi)}}{\psi},$ 
with $\displaystyle\rho(x,\xi):=-\frac{1}{|\psi_x|^2}\psi_{x^k}(\psi_{x^k})_{(\xi)}$; $\hat{r}_t^\alpha:=\hat{r}(x_t^{\alpha,x},\xi_t^{\alpha,\xi})$, where $\displaystyle\hat{r}(x,\xi):=\pxsops$;
\item$\pi_t^\alpha:=\pi(x_t^{\alpha,x},\xi_t^{\alpha,\xi})$, where $\displaystyle\pi^k(x,\xi):=\frac{1}{2\gamma}\Big(1-\frac{\psi}{2\lambda}\Big)\bigg[\frac{\psi_{(\xi)}}{\psi}\psi_{(\sigma_k)}+\beta(\xi,\sigma_k)\bigg],\forall  k=1,...,d_1$; $\displaystyle\hat\pi_t^\alpha:=-2\pi^\alpha_t\tilde\xi_t^{\alpha,0}=-2\pi^\alpha_t\int_0^t\pi_s^{\alpha}dw_s$;
\item$P_t^\alpha:=P(x_t^{\alpha,x},\xi_t^{\alpha,\xi})$, where $\displaystyle P_{jk}(x,\xi):=\frac{(\psi_{x^j})_{(\xi)}\psi_{x^k}-(\psi_{x^k})_{(\xi)}\psi_{x^j}}{|\psi_x|^2}$, $\forall j,k=1,...,d$; $\hat P_t^\alpha:=0.$
\end{itemize}

Let $\tau_1^\delta=\tau^{\alpha,x}_{D^\lambda_\delta}$. When the constant $\lambda$ is sufficiently small, for all $x\in D^\lambda_\delta$ and $\xi\in\Rd$, we have
\begin{enumerate}
\item For each $\alpha\in\mathfrak A$, $\mathrm{B}^\kappa_1(x^{\alpha,x}_t,\xi^{\alpha,\xi}_t)$ is a local supermartingale on $[0,\tau_1^\delta]$ for each $\kappa\in[0,1+\kappa_1]$, where $\kappa_1=\kappa_1(K_0,d,d_0,D,\lambda,\delta)$ is a sufficiently small positive constant;
\item $\displaystyle{\sup_{\alpha\in\mathfrak A}E^\alpha_{x,\xi}\int_0^{\tau_1^\delta}\Big(|\xi_t|^2+\pxtsops\Big) dt\le N\mathrm{B}_1(x,\xi)}$;
\item $\displaystyle{\sup_{\alpha\in\mathfrak A}E^\alpha_{\xi}\sup_{t\le\tau_1^\delta}|\xi_t|^2\le N\mathrm{B}_1(x,\xi)}$;
\item $\displaystyle{\sup_{\alpha\in\mathfrak A}E^\alpha_{0}\sup_{t\le\tau_1^\delta}|\tilde\xi_t|^2\le N\mathrm{B}_1(x,\xi)}$;
\item $\displaystyle{\sup_{\alpha\in\mathfrak A}E^\alpha_{0}\sup_{t\le\tau_1^\delta}|\eta_t|\le N\mathrm{B}_1(x,\xi)}$;
\item $\displaystyle{\sup_{\alpha\in\mathfrak A}E^\alpha_{0}\Big(\int_0^{\tau_1^\delta}|\eta_t|^2 dt\Big)^{1/2}\le N\mathrm{B}_1(x,\xi)}$; 
\end{enumerate}
where $N$ is a constant depending on $K_0,d,d_1,D$ and $\lambda$.
\end{lemma}

\begin{proof}

We drop the superscript $\alpha$ throughout the proof. We may drop the argument $x$ or $x_t$ when this will cause no confusion. Also, keep in mind that the constant $K\in[1,\infty)$ depends only on $K_0,d,d_1,D$, while the constant $N\in[1,\infty)$ depends on $K_0,d,d_1,D$ and $\lambda$.

We first notice that (\ref{domain}), there exists a small positive constant $\mu$ depending on the domain $D$, such that $|\psi_x|\ge1/2$ in $D^\mu$. By choosing $\lambda$ small than $\mu$, we may assume that $|\psi_x|\ge1/2$ in $D^\lambda$.

By It\^o's formula,  we have
\begin{align*}
d\pxt=&\Big[\psi_{(\xi_t)(\sigma_k)}+ r_t\psk+\psi_{(P_t\sigma_k)}\Big]dw_t^k\\
+&\Big[(L\psi)_{(\xi_t)}+2 r_tL\psi-\psi_{(\sigma\pi)}+\sum_k(\psi_{xx}\sigma_k,P_t\sigma_k)\Big]dt.
\end{align*}
Due to our choices of $r$ and $P$, we have
$$\sum_k(\psi_{xx}\sigma_k,P\sigma_k)=\operatorname{tr}(\sigma\sigma^*\psi_{xx}P)=0,$$
\begin{equation}\label{pproperty}
\psi_{(\xi)(\sigma_k)}+ \rho\psi_{(\sigma_k)}+\psi_{(P\sigma_k)}=0.
\end{equation}
Thus
\begin{equation}
d\pxt=\frac{\pxt}{\psi}\psk dw_t^k+\Big[(L\psi)_{(\xi_t)}+2 r_tL\psi-\psi_{(\sigma\pi_t)}\Big]dt.\label{4l}
\end{equation}

Let  $\bar\sigma:=r\sigma+P\sigma$ and $\bar b:=2 r b$.
Again, by It\^o's formula,
\begin{equation*}
d\mathrm{B}_1(x_t, \xi_t)=\Gamma_1(x_t, \xi_t)dt+\Lambda_1(x_t,\xi_t) dw_t,
\end{equation*}
with
\allowdisplaybreaks\begin{align*}
\Gamma_1(x, \xi)=&\Big(\beta|\xi|^2+\pxsop\Big)L\gamma\\
&+\gamma\bigg\{|\xi|^2L\beta+\beta\Big[2(\xi,\bar b-\sigma\pi)+\|\bar\sigma\|^2\Big]+2\beta_{(\sigma_k)}(\xi,\bar\sigma_k)\\
&+2\pxop\Big[(L\psi)_{(\xi)}+2\rho L\psi\Big]+3\pxsops L\psi-2\pxop\psi_{(\sigma\pi)}\bigg\}\\
&+\gamma_{(\sigma_k)}\bigg\{2\beta(\xi,\bar\sigma_k)+\beta_{(\sigma_k)}|\xi|^2+\pxsops\psk\bigg\}\\
=&I_1+I_2+I_3+\gamma(J_1+J_2+J_3+J_4),
\end{align*}
where
\allowdisplaybreaks\begin{align*}
I_1=&\Big(\beta|\xi|^2+\pxsop\Big)L\gamma,\\
I_2=&-2\gamma\beta(\xi,\sigma\pi)-2\gamma\pxop\psi_{(\sigma\pi)}+\gamma_{(\sigma_k)}\Big[2\beta(\xi,\sigma_k)\pxop+\pxsops\psk\Big],\\
I_3=&\gamma_{(\sigma_k)}\bigg\{2\beta(\xi,\rho\sigma_k+P\sigma_k)+\beta_{(\sigma_k)}|\xi|^2\bigg\},\\
J_1=&|\xi|^2L\beta+3\pxsops L\psi,\qquad J_2=\beta\Big[2(\xi,\bar b)+\|\bar\sigma\|^2\Big],\\
J_3=&2\beta_{(\sigma_k)}(\xi,\bar\sigma_k),\qquad J_4=2\pxop\Big[(L\psi)_{(\xi)}+2\rho L\psi\Big].
\end{align*}
In order that $B_1(x_t,\xi_t)$ is a local supermartingale, we need that $\Gamma_1(x,\xi)\le 0$. To this end, we estimate from $I_1$ to $J_4$ term by term:
\begin{align*}\allowdisplaybreaks
I_1\le&-\frac{1}{2}|L\psi||\xi|^2-\frac{1}{4\lambda}\psk^2|\xi|^2,\\
I_2
=&-\Big(1-\frac{\psi}{2\lambda}\Big)\beta^2\psk^2\le0,\\
I_3\le&|\psk|\Big(K|\xi|^2+\frac{1}{8\lambda}|\psk||\xi|^2\Big)\\
\le&\Big(4\lambda K^2+\frac{1}{8\lambda}\psk^2\Big)|\xi|^2+\frac{1}{8\lambda}\psk^2|\xi|^2\\
\le&\lambda K|\xi|^2+\frac{1}{4\lambda}\psk^2|\xi|^2,\\
J_1\le&\frac{1}{8\lambda}\Big(\frac{1}{2}L\psi-\frac{1}{4\lambda}\psk^2\Big)|\xi|^2+3\pxsops L\psi\\
=&-\frac{1}{16\lambda}|L\psi||\xi|^2-\frac{1}{32\lambda^2}\psk^2|\xi|^2-3\pxsops|L\psi|,\\
J_2\le&2\bigg[2|\xi|\Big(K|\xi|+2\frac{|\psi_{(\xi)}|}{\psi}|b|\Big)+\sum_k\Big(K|\xi|+\frac{|\psi_{(\xi)}|}{\psi}|\sigma_k|\Big)^2\bigg]\\
\le&K|\xi|^2+\pxsops|b|^2+4\pxsops\|\sigma\|^2\\
\le&K|\xi|^2+\pxsops|L\psi|,\\
J_3\le&2\frac{1}{8\lambda}\sum_k|\psk||\xi|\Big(K|\xi|+\frac{|\psi_{(\xi)}|}{\psi}|\sigma_k|\Big)\\
\le&\frac{1}{4\lambda}\bigg[(1/8+K\psk^2)|\xi|^2+\frac{1}{16\lambda}\psk^2|\xi|^2+8\lambda\pxsops\|\sigma\|^2\bigg]\\
\le&\frac{1}{32\lambda}|\xi|^2+\frac{16\lambda K+1}{64\lambda^2}\psk^2|\xi|^2+\frac{1}{2}\pxsops|L\psi|,\\
J_4\le&2\frac{|\psi_{(\xi)}|}{\psi}K|\xi|\le K|\xi|^2+\pxsops.
\end{align*}
Collecting our estimates above we see that, for all $(x,\xi)\in D_\delta^\lambda\times\Rd$,
\begin{align*}
\Gamma_1(x,\xi)\le& \Big(-\frac{1}{2}|L\psi|+\lambda K\Big)|\xi|^2+\gamma\bigg[\Big(-\frac{1}{16\lambda}|L\psi|+K+\frac{1}{32\lambda}\Big)|\xi|^2\\
&+\Big(-\frac{1}{32\lambda^2}+\frac{16\lambda K+1}{64\lambda^2}\Big)\psk^2|\xi|^2+\Big(-\frac{3}{2}|L\psi|+1\Big)\pxsops\bigg]
\end{align*}
By choosing sufficiently small positive $\lambda$, we get
\begin{equation}\label{Gamma1}
\Gamma_1(x,\xi)\le -(1/4)|\xi|^2-(\gamma/2)\pxsops\le-(1/4)|\xi|^2-(\lambda^2/2)\pxsops.
\end{equation}
It follows that $\mathrm{B}_1(x_t,\xi_t)$ is a local supermartingale on $[0,\tau_1^\delta]$.

For each $\kappa<1$, $\B_1^\kappa(x_t,\xi_t)$ is a local supermartingale since the power function $x^\kappa$ is concave. If $\kappa>1$, by It\^o's formula, we have
\begin{align*}
d\B_1^\kappa(x_t,\xi_t)=&\kappa \B_1^{\kappa-1}(x_t,\xi_t)\Lambda_1(x_t,\xi_t)dw_t+\Delta_1(x_t,\xi_t)dt,
\end{align*}
where
\begin{align*}
\Delta_1(x,\xi)=\kappa \B_1^{\kappa-1}(x,\xi)\Gamma_1(x,\xi)+\frac{\kappa(\kappa-1)}{2}\B_1^{\kappa-2}(x,\xi)\|\Lambda(x,\xi)\|^2.
\end{align*}
Notice that
\begin{align*}
\kappa \B_1^{\kappa-1}(x,\xi)\Gamma_1(x,\xi)\le&-(\lambda^{2\kappa-2}/4)|\xi|^{2\kappa},\\
\frac{\kappa(\kappa-1)}{2}\B_1^{\kappa-2}(x,\xi)\|\Lambda_1(x,\xi)\|^2\le&(\kappa-1)C(K,\lambda,\delta)|\xi|^{2\kappa},
\end{align*}
therefore $\Delta_1(x,\xi)<0$ when $\kappa-1$ is sufficiently small. Thus (1) is proved.

From (\ref{Gamma1}), by letting $\lambda_0=\lambda^2/2$, we have
$$
\Gamma_1(x,\xi)+\lambda_0\Big(|\xi|^2+\pxsops\Big)\le 0,\qquad\forall (x,\xi)\in D_\delta^\lambda\times\Rd.
$$
Therefore, 
\begin{align*}
\lambda_0 E\int_0^{\tau_1^\delta}\Big(|\xi_t|^2+\pxtsops \Big)dt\le&-E\int_0^{\tau_1^\delta}\Gamma_1(x_t,\xi_t)dt\le\mathrm{B}_1(x,\xi),
\end{align*}
which proves (2).

To show (3), by Davis inequality, for $\tau_n=\tau_1^\delta\wedge\inf\{t\ge0:|\xi_t|\ge n\}$,
\begin{align*}
E\sup_{t\le\tau_n}|\xi_t|^2\le& |\xi|^2+\int_0^{\tau_n}\Big(2|\xi_t|\cdot|\bar b_t-\sigma\pi_t|+\|\bar\sigma\|^2\Big)dt+6E\Big(\int_0^{\tau_n}|(\xi_t,\bar\sigma_t)|^2dt\Big)^{\frac{1}{2}}\\
\le&|\xi|^2+NE\int_0^{\tau_n}\Big(|\xi_t|^2+\frac{\psi_{(\xi_t)}^2}{\psi^2}\Big)dt+E\bigg(\int_0^{\tau_n} N|\xi_t|^2\Big(|\xi_t|^2+\frac{\psi_{(\xi_t)}^2}{\psi^2}\Big)dt\bigg)^{\frac{1}{2}}\\
\le&N\mathrm{B}_1(x,\xi)+E\bigg[\sup_{t\le\tau_n}|\xi_t|\Big(\int_0^{\tau_n} N\Big(|\xi_t|^2+\frac{\psi_{(\xi_t)}^2}{\psi^2}\Big)dt\Big)^{\frac{1}{2}}\bigg]\\
\le&N\mathrm{B}_1(x,\xi)+\frac{1}{2}E\sup_{t\le\tau_n}|\xi_t|^2,
\end{align*}
which implies that
$$
E\sup_{t\le\tau_n}|\xi_t|^2\le N\mathrm{B}_1(x,\xi).
$$
Now  (3) is obtained by first letting $n\rightarrow\infty$ and then taking the supremum with respect to $\alpha$.

To show (4) it suffices to notice that
$$E\sup_{t\le\tau_1^\delta}|\tilde\xi_t|^2\le 4 E\int_0^{\tau_1^\delta}|\pi_t|^2dt\le NE\int_0^{\tau_1^\delta} \Big(|\xi_t|^2+\frac{\psi_{(\xi_t)}^2}{\psi^2}\Big)dt.$$

Now we estimate the moment of the second quasiderivative $\eta_t$. Based on our definition, we have
$$
d\eta_t=G_tdw_t+H_tdt,
$$
with
\begin{equation*}
\|G_t\|\le N|\xi_t|\bigg(|\xi_t|+\frac{|\psi_{(\xi_t)}|}{\psi}\bigg),\qquad|H_t|\le N\bigg(|\xi_t|^2+\pxtsops+|\hat\pi_t|\bigg).
\end{equation*}
Let $\gamma_t=\gamma(x_t)$. By It\^o's formula we have
$$d\big(e^{2\gamma_t}|\eta_t|^2\big)=\Upsilon(x_t,\xi_t,\eta_t)dw_t+\Theta(x_t,\xi_t,\eta_t)dt,$$
where
\begin{align*}
\Theta(x,\xi,\eta)=&e^{2\gamma_t}\bigg\{2|\eta|^2\Big[\Big(1-\frac{\psi}{2\lambda}\Big)L\psi-\frac{1}{4\lambda}\psk^2+\Big(1-\frac{\psi}{2\lambda}\Big)^2\psk^2\Big]\\
&+2(\eta,H)+\|G\|^2+4(\eta,G)\psk\Big(1-\frac{\psi}{2\lambda}\Big)\bigg\}\\
\le&e^{2\gamma_t}\bigg[-|\eta|^2+N\big(|\eta|+|\xi|^2\big)\bigg(|\xi|^2+\pxsops+|\hat\pi|\bigg)\bigg].
\end{align*}
Then for any bounded stopping time $\tau$ we have
$$E\big(e^{2\gamma_\tau}|\eta_\tau|^2\big)+E\int_0^\tau|\eta_t|^2dt\le NE\int_0^\tau e^{2\gamma_t}\big(|\eta_t|+|\xi_t|^2\big)\bigg(|\xi_t|^2+\pxtsops+|\hat\pi_t|\bigg)dt.$$
Let $\tau_n=\tau_1^\delta\wedge\inf\{t\ge0:e^{\gamma_t}|\eta_t|\ge n\}$. Recall that $\eta=0$. By Theorem III.6.8 in \cite{MR1311478}, we have
\begin{align*}
E\sup_{t\le\tau_n}(e^\gamma|\eta_t|)\le&3E\bigg[\int_0^{\tau_n}N e^{2\gamma_t}\big(|\eta_t|+|\xi_t|^2\big)\bigg(|\xi_t|^2+\pxtsops+|\hat\pi_t|\bigg)dt\bigg]^{1/2}\\
\le&NE\bigg\{\sup_{t\le\tau_n}\Big[e^{\gamma_t}(|\eta_t|^{1/2}+|\xi_t|)\Big]\bigg[\int_0^{\tau_n}\Big(|\xi_t|^2+\pxtsops+|\hat\pi_t|\Big)dt\bigg]^{1/2}\bigg\}\\
\le&\frac{1}{2}E\sup_{t\le\tau_n}(e^{\gamma_t}|\eta_t|)+NE\sup_{t\le\tau_n}|\xi_t|^2+NE\int_0^{\tau_n}\Big(|\xi_t|^2+\pxtsops\Big)dt\\
&+2E\bigg(\sup_{t\le\tau_n}|\tilde\xi_t|\cdot\int_0^{\tau_n}|\pi_t|dt\bigg)\\
\le&\frac{1}{2}E\sup_{t\le\tau_n}(e^{\gamma_t}|\eta_t|)+NB_1(x,\xi).
\end{align*}
It follows that
\begin{align*}
E\sup_{t\le\tau_n}|\eta_t|\le E\sup_{t\le\tau_n}(e^{\gamma_t}|\eta_t|)\le &N\B_1(x,\xi),\\
E\bigg(\int_0^{\tau_n}|\eta_t|^2dt\bigg)^{1/2}\le &N\B_1(x,\xi).
\end{align*}
Letting $n\rightarrow\infty$ and then taking the supremum over $\mathfrak A$, (5) and (6) are proved.
\end{proof}

\begin{lemma} \label{l2}
In $D_{\lambda^2}\times\Rd$, let
$$
\mathrm{B}_2(x,\xi)=\lambda^{3\theta}\psi^{1-2\theta}\bigg[K_1|\xi|^2+\pxsop\bigg],
$$
where $\theta\in(0,1/3)$ and $K_1\in[1,\infty)$ are constants depending on $K_0$,$d,d_1$,$D$, to be determined in the proof.

For each $\alpha$, we define the first and second quasiderivatives by (\ref{itoxi}) and (\ref{itoeta}), in which 
\begin{itemize}
\item $r_t^\alpha:=r(x_t^{\alpha,x},\xi_t^{\alpha,\xi})$, where $\displaystyle r(x,\xi):=\theta\frac{\psi_{(\xi)}}{\psi}$; $\hat{r}_t^\alpha:=0$;
\item$\pi_t^\alpha:=\pi(x_t^{\alpha,x},\xi_t^{\alpha,\xi})$, where $\displaystyle\pi^k(x,\xi):=\frac{\nu}{\psi^2}\bigg[K_1\psi(\xi,\sigma_k)+\psi_{(\xi)}\psi_{(\sigma_k)}\bigg],\forall  k=1,...,d_1$, and $\displaystyle\nu=\frac{\theta(1-2\theta)^2}{2(1-3\theta)}$; $\displaystyle\hat\pi_t^\alpha:=-2\pi_t^\alpha\tilde\xi_t^{\alpha,0}=-2\pi_t^\alpha\int_0^t\pi^\alpha_sdw_s$;
\item$P_t^\alpha=\hat P_t^\alpha:=0.$
\end{itemize}
 
Let $\tau_2=\tau^{\alpha,x}_{D_{\lambda^2}}$. For $x\in D^\lambda_\delta$ and $\xi\in\Rd$, we have
\begin{enumerate}
\item For each $\alpha\in\mathfrak A$, $\mathrm{B}^\kappa_2(x^{\alpha,x}_t,\xi^{\alpha,\xi}_t)$ is a local supermartingale on $[0,\tau_2]$ for each $\kappa\in[0,1+\kappa_2]$, where $\kappa_2=\kappa_2(K_0,d,d_0,D,\lambda)$ is a sufficiently small positive constant;
\item $\displaystyle{\sup_{\alpha\in\mathfrak A}E^\alpha_{\xi}\int_0^{\tau_2}|\xi_t|^2dt\le N\mathrm{B}_2(x,\xi)}$;
\item $\displaystyle{\sup_{\alpha\in\mathfrak A}E^\alpha_{\xi}\sup_{t\le\tau_2}|\xi_t|^2\le N\mathrm{B}_2(x,\xi)}$;
\item $\displaystyle{\sup_{\alpha\in\mathfrak A}E^\alpha_{0}\sup_{t\le\tau_2}|\tilde\xi_t|^2\le N\mathrm{B}_2(x,\xi)}$;
\item $\displaystyle{\sup_{\alpha\in\mathfrak A}E^\alpha_{0}\sup_{t\le\tau_2}|\eta_t|\le N\mathrm{B}_2(x,\xi)}$;
\item $\displaystyle{\sup_{\alpha\in\mathfrak A}E^\alpha_{0}\bigg(\int_0^{\tau_2}|\eta_t|^2 dt\bigg)^{1/2}\le N\mathrm{B}_2(x,\xi)}$; 
\end{enumerate}
where $N$ is a constant depending on $K_0,d,d_1,D$ and $\lambda$.

\end{lemma}

\begin{proof}
Again, we drop the superscript $\alpha$ throughout the proof and may drop the argument $x$ or $x_t$ when this will cause no confusion. Also, keep in mind that the constant $K\in[1,\infty)$ depends only on $K_0,d,d_1,D$, while the constant $N\in[1,\infty)$ depends on $K_0,d,d_1,D$ and $\lambda$.

Notice that the factor $\lambda^{3\theta}$ is a constant, so it doesn't hurt to ignore this factor throughout the proof of this lemma.

By It\^o's formula, we have
$$d\mathrm B_2(x_t,\xi_t)=\Gamma_2(x_t,\xi_t)dt+\Lambda_2(x_t,\xi_t)dw_t,$$
with
\begin{align*}
\Gamma_2(x,\xi)=&K_1\psi^{1-2\theta}\Big[2(\xi,2rb-\sigma\pi)+\|r\sigma\|^2\Big]\\
&+K_1|\xi|^2\Big[(1-2\theta)\psi^{-2\theta}L\psi-\theta(1-2\theta)\psi^{-2\theta-1}\psk^2\Big]\\
&+K_12(1-2\theta)(\xi,r\sigma_k)\psi^{-2\theta}\psk\\
&+\psi^{-2\theta}\Big\{2\psi_{(\xi)}\Big[(L\psi)_{(\xi)}+2rL\psi-\psi_{(\sigma_k)}\pi^k\Big]+\Big[\psi_{(\xi)(\sigma_k)}+r\psi_{(\sigma_k)}\Big]^2\Big\}\\
&+\psi_{(\xi)}^2\Big[-2\theta\psi^{-2\theta-1}L\psi+\theta(2\theta+1)\psi^{-2\theta-2}\psk^2\Big]\\
&-4\theta\psi^{-2\theta-1}\psk\psi_{(\xi)}\Big[\psi_{(\xi)(\sigma_k)}+r\psk\Big]\\
=&\psi^{-2\theta}I_1+\psi^{-2\theta-1}I_2+\psi^{-2\theta-2}I_3,
\end{align*}
where
\begin{align*}
I_1=&4K_1\theta(\xi,b)\px+K_1(1-2\theta)|\xi|^2L\psi+2\px(L\psi)_{(\xi)}+(\psi_{(\xi)(\sigma_k)})^2,\\
I_2=&K_1\theta^2\px^2\|\sigma\|^2-K_1\theta(1-2\theta)|\xi|^2\psk^2+2\theta\px^2L\psi-2\theta\psi_{(\xi)(\sigma_k)}\px\psk,\\
I_3=&-2K_1^2\nu\psi^2|(\xi,\sigma_k)|^2+\Big[-4K_1\nu+2K_1\theta(1-2\theta)\Big]\psi(\xi,\sigma_k)\px\psk\\
&+\Big[-2\nu+\theta^2+\theta(2\theta+1)-4\theta^2\Big]\px^2\psk^2.
\end{align*}
We claim that $I_1$, $I_2$ and $I_3$ are all non-positive for suitable $K_1$, $\theta$ and $\nu$.
First, to estimate $I_1$, we notice that
$$I_1\le\Big[4K_1\theta|\psi_x|_{0,D}|b|_{0,D}+K-K_1(1-2\theta)|L\psi|\Big]|\xi|^2$$
If $\theta$ is sufficient small such that $4\theta|\psi_x|_{0,D}|b|_{0,D}\le1/4$, then we have
\begin{align*}
I_1\le&\Big[(1/4)K_1+K-(1/3)K_1|L\psi|\Big]|\xi|^2\le\Big[K-(1/12)K_1|L\psi|\Big]|\xi|^2
\end{align*}
Therefore, $I_1\le-K|\xi|^2$ for sufficiently large $K_1\ge24K$.

Next, to estimate $I_2$, we observe that
$$-2\psi_{(\xi)(\sigma_k)}\px\psk\le K|\xi|\|\sigma\||\px||\sigma^*\psi_x|\le\px^2\|\sigma\|^2+K|\xi|^2|\sigma^*\psi_x|^2.$$
It follows that
\begin{align*}
I_2\le&\theta\Big[(K_1\theta+1)\|\sigma\|^2+2L\psi\Big]\px^2+\theta\Big[K-K_1(1-2\theta)\Big]|\xi|^2\psk^2.
\end{align*}
By first choosing sufficiently large $K_1$ such that $K-(1/3)K_1\le0$ and then sufficiently small $\theta$ such that $K_1\theta\le1$, we get $I_2\le0$.

To estimate $I_3$, by letting $\mathfrak a_k=\psi(\xi,\sigma_k)$ and $\mathfrak b_k=\px\psk$, we can rewrite $I_1$ as
$$-2K_1^2\nu \mathfrak a_k^2+\Big[-4K_1\nu+2K_1\theta(1-2\theta)\Big]\mathfrak a_k\mathfrak b_k+\Big[-2\nu+(1-\theta)\theta\Big]\mathfrak b_k^2.$$
In order to make the above quadratic form non-positive, it suffices to find a constant $\nu>0$ such that the discriminant equals zero, which yields that
$$\nu=\theta(1-2\theta)^2/[2(1-3\theta)]>0.$$
This is exactly how $\nu$ is defined in the statement of the lemma.

Collecting the estimates above we see that, if we pick the constants $K_1=24K$ and $\theta=\min\{1/3,1/K_1, 1/(16|\psi_x|_{0,D}|b|_{0,D})\}$, then 
\begin{equation}\label{Gamma2}
\Gamma_2(x,\xi)\le -\psi^{-2\theta}K|\xi|^2<0, \qquad\forall(x,\xi)\in D_{\lambda^2}\times\Rd.
\end{equation}
Thus $\mathrm{B}_2(x_t,\xi_t)$ is a local supermartingale on $[0,\tau_2]$.

Properties (1)-(4) can be verified by almost repeating the proof of Properties (2)-(4) in Lemma \ref{l1}. To prove (5) and (6) we apply It\^o's formula to $\exp(2\sqrt{\psi(x_t)})|\eta_t|^2$ and then mimic the proof of Properties (5) and (6) in Lemma \ref{l1}.

\end{proof}

\begin{lemma}\label{l3}
For sufficiently small $\lambda$, we have,
\begin{align}
\B_1(x,\xi)\ge& 4\B_2(x,\xi)\mbox{ on }\{x:\psi(x)=\lambda\}\times\Rd,\label{b1ge4b2}\\
\B_2(x,\xi)\ge& 4\B_1(x,\xi)\mbox{ on }\{x:\psi(x)=\lambda^2\}\times\Rd\label{b2ge4b1}.
\end{align}
\end{lemma}
\begin{proof}
Direct substitution leads to
\begin{equation*}
\B_1(x,\xi)=\left\{
\begin{array}{ll}
\displaystyle\lambda\Big(\frac{3}{4}+\lambda\Big)\bigg[(35/32)|\xi|^2+\pxsop\bigg],&\mbox{if }\psi=\lambda,\\
\displaystyle\lambda^2\Big(2-\frac{\lambda}{4}\Big)\bigg[\Big(1+\frac{\lambda-\lambda^2/4}{8}\Big)|\xi|^2+\pxsop\bigg],&\mbox{if }\psi=\lambda^2;
\end{array}
\right.  
\end{equation*}
\begin{equation*}
\B_2(x,\xi)=\left\{
\begin{array}{lll}
\displaystyle\lambda^{1+\theta}\bigg[K_1|\xi|^2+\pxsop\bigg],\qquad\qquad\qquad\qquad\ \ &\mbox{if }\psi=\lambda,\\
\displaystyle\lambda^{2-\theta}\bigg[K_1|\xi|^2+\pxsop\bigg],&\mbox{if }\psi=\lambda^2.
\end{array}
\right.  
\end{equation*}
Recall that $K_1$ and $\theta$ don't depend on $\lambda$, and $\theta\in(0,1/3)$. Therefore, (\ref{b1ge4b2}) and (\ref{b2ge4b1}) are true for sufficiently small $\lambda$.

\end{proof}

With quasiderivatives in both of the subdomains $D_\delta^\lambda$ and $D_{\lambda^2}$, we next construct quasiderivatives in $D_\delta$. Roughly speak, we glue the quasiderivatives constructed in Lemmas \ref{l1} and \ref{l2}.

Let $x,y,z\in D_\delta$ and $\xi,\eta\in\Rd$. We start from defining stopping times as follows:
\begin{align*}
\tau^{\alpha,x}_\delta=&\inf\{t\ge0: x_t^{\alpha,x}\notin D_\delta\},\\
\tau^{\alpha,x}_{-1}=&0,\\
\tau^{\alpha,x}_0=&\inf\{t\ge0:\psi(x_t^{\alpha,x})\le \lambda^2\},\\
\tau^{\alpha,x}_1=&\tau_\delta^{\alpha,x}\wedge\inf\{t\ge\tau^{\alpha,x}_0:\psi(x_t^{\alpha,x})\ge \lambda\},
\end{align*}
and recursively, for $n\in\mathbb N$,
\begin{align*}
\tau^{\alpha,x}_{2n}=&\tau_\delta^{\alpha,x}\wedge\inf\{t\ge\tau_{2n-1}^{\alpha,x}:\psi(x_t^{\alpha,x})\le \lambda^2\},\\
\tau^{\alpha,x}_{2n+1}=&\tau_\delta^{\alpha,x}\wedge\inf\{t\ge\tau_{2n}^{\alpha,x}:\psi(x_t^{\alpha,x})\ge \lambda\}.
\end{align*}
Then for each $t\in[0,\tau_\delta^{\alpha,x})$, the auxiliary processes $r_t^{\alpha,x}$, $\hat r_t^{\alpha,x}$, $\pi_t^{\alpha,x}$, $\hat\pi_t^{\alpha,x}$, $P_t^{\alpha,x}$ and $\hat P_t^{\alpha,x}$ are defined by Lemma \ref{l1} when $t\in[\tau^{\alpha,x}_{2n-2},\tau^{\alpha,x}_{2n-1})$, and by Lemma \ref{l2} when $t\in[\tau^{\alpha,x}_{2n-1},\tau^{\alpha,x}_{2n})$. Therefore, on $[0,\tau_\delta^{\alpha,x})$, we can define almost surely continuous processes $\xi_t^{\alpha,x}$, $\eta_t^{\alpha,x}$, $\tilde\xi_t^{\alpha}$ and $\tilde\eta_t^{\alpha}$ based on (\ref{itoxi})-(\ref{itoetaa}) by letting the initial points in each time subinterval be the terminal points in the previous time subinterval. Similarly, on $[0,\tau_\delta^{\alpha,x}\wedge\bar\tau_\delta^{\alpha,y}(\epsilon))$, we can define almost surely continuous processes $y_t^{\alpha,y}(\epsilon)$ by (\ref{itoy}), and on $[0,\tau_\delta^{\alpha,x}\wedge\hat\tau_\delta^{\alpha,z}(\epsilon))$ almost surely continuous processes $z_t^{\alpha,z}(\epsilon)$ by (\ref{itoz}). Note that based on our construction, the processes defined by (\ref{itoy})-(\ref{itoetaa}) have unique representation in each time subinterval.

For convenience of notation, on $D_\delta$, we define
\begin{align*}
\overline{\mathrm B}(x,\xi)=&\mathbbm1_{x\in D_\delta^\lambda}\B_1(x,\xi)+\mathbbm1_{x\in \bar D_{\lambda^2}}\B_2(x,\xi),\\
\underline\B(x,\xi)=&\left\{
\begin{array}{ll}
\B_1(x,\xi)&\mbox{ in }D_\delta^{\lambda^2}\\
\min\{\B_1(x,\xi),\B_2(x,\xi)\}&\mbox{ in }\bar D_{\lambda^2}^\lambda\\
\B_2(x,\xi)&\mbox{ in }D_{\lambda}.
\end{array}
\right.
\end{align*}

From now on, the stochastic processes $r_t^{\alpha,x}$, $\hat r_t^{\alpha,x}$, $\pi_t^{\alpha,x}$, $\hat\pi_t^{\alpha,x}$, $P_t^{\alpha,x}$, $\hat P_t^{\alpha,x}$, $\xi_t^{\alpha,x}$, $\eta_t^{\alpha,x}$, $\tilde\xi_t^{\alpha}$, $\tilde\eta_t^{\alpha}$, $y_t^{\alpha,y}(\epsilon)$ and $z_t^{\alpha,z}(\epsilon)$ are supposed to be defined in the way mentioned above.

\begin{lemma}\label{applicable}
For each $x\in D_\delta$, $\xi\in\Rd$, we have
\begin{enumerate}
\item $\displaystyle{\sup_{\alpha\in\mathfrak A}E^\alpha_{\xi}\sup_{t\le\tau_\delta}|\xi_t|^2\le N\overline{\mathrm B}(x,\xi)}$;
\item $\displaystyle{\sup_{\alpha\in\mathfrak A}E^\alpha_{0}\sup_{t\le\tau_\delta}|\tilde\xi_t|^2\le N\overline{\mathrm B}(x,\xi)}$;
\item $\displaystyle{\sup_{\alpha\in\mathfrak A}E^\alpha_{x,\xi}\int_0^{\tau_\delta}\Big(|\xi_t|^2+\pxtsops\Big) dt\le N\overline{\mathrm B}(x,\xi)}$;
\item $\displaystyle{\sup_{\alpha\in\mathfrak A}E^\alpha_\xi|\xi_\gamma|^{2(1+l)}\le N\overline{\B}^{1+l}(x,\xi)}$, with $l=\kappa_1\wedge\kappa_2$, for any $\gamma^{\alpha,x}\le\tau_\delta^{\alpha,x}$;
\item $\displaystyle{\sup_{\alpha\in\mathfrak A}E^\alpha_{0}\sup_{t\le\tau_\delta}|\eta_t|\le N\overline{\mathrm B}(x,\xi)}$;
\item $\displaystyle{\sup_{\alpha\in\mathfrak A}E^\alpha_{0}\Big(\int_0^{\tau_\delta}|\eta_t|^2 dt\Big)^{1/2}\le N\overline{\mathrm B}(x,\xi)}$;
\item $\displaystyle{\sup_{\alpha\in\mathfrak A}E^\alpha_{x,\xi}\underline{\mathrm B}(x_\gamma,\xi_\gamma)\le 2\overline{\mathrm B}(x,\xi)}$, for any $\gamma^{\alpha,x}\le\tau_\delta^{\alpha,x}$.
\end{enumerate}
where $N$ is a constant depending on $K_0,d,d_1,D$ and $\lambda$. Meanwhile, the conclusions in Lemmas \ref{thm1}, \ref{thm2} and \ref{markovtime} are all true for these processes with $U=D_\delta$.
\end{lemma}

\begin{proof}
It suffices to prove the uncontrolled version since the righthand sides of the inequalities are independent of $\alpha$. Let
\begin{equation*}
\B_n(x,\xi)=\left\{
\begin{array}{ll}
\B_1(x,\xi),&\mbox{for odd $n$}\\
\B_2(x,\xi),&\mbox{for even $n$.}
\end{array}
\right.
\end{equation*}
Suppose that the constant $N$ in Lemma \ref{l1}(3) and Lemma \ref{l2}(3) are the same by choosing the larger one. For $n=-1,0,1,2,\dots$ By the strong Markov property and Lemma \ref{l3}, we have,
\begin{align*}
&E\Big[\mathbbm1_{\tau_{n+1}<\tau_\delta}\B_{n+1}(x_{\tau_{n+1}},\xi_{\tau_{n+1}})+(1/N)\mathbbm1_{\tau_n<\tau_\delta}\sup_{\tau_n\le t<\tau_{n+1}}\big||\xi_{t\wedge\tau_\delta}|^2-|\xi_{\tau_n}|^2\big|\Big]\\
=&E\Big\{E\Big[\mathbbm1_{\tau_{n+1}<\tau_\delta}\B_{n+1}(x_{\tau_{n+1}},\xi_{\tau_{n+1}})+(1/N)\mathbbm1_{\tau_n<\tau_\delta}\sup_{\tau_n\le t<\tau_{n+1}}\big||\xi_{t\wedge\tau_\delta}|^2-|\xi_{\tau_n}|^2\big|\Big|\mathcal F_{\tau_n}\Big]\Big\}\\
\le&E\mathbbm1_{\tau_n<\tau_\delta}2\B_{n+1}(x_{\tau_n},\xi_{\tau_n})\\
\le&E\mathbbm1_{\tau_n<\tau_\delta}\B_{n}(x_{\tau_n},\xi_{\tau_n}).
\end{align*}
Adding the inequalities over $n=-1,0,\dots,m,$ and canceling duplicate terms, we have
\begin{equation*}
\begin{gathered}
E\Big[\mathbbm1_{\tau_{m+1}<\tau_\delta}\B_{m+1}(x_{\tau_{m+1}},\xi_{\tau_{m+1}})+(1/N)\sum_{n=-1}^{m}\mathbbm1_{\tau_n<\tau_\delta}\sup_{\tau_n\le t<\tau_{n+1}}\big||\xi_{t\wedge\tau_\delta}|^2-|\xi_{\tau_n}|^2\big|\Big]\\
\le \mathbbm1_{x\in D_\delta^\lambda}\B_1(x,\xi)+\mathbbm1_{x\in D_{\lambda^2}}\B_2(x,\xi).
\end{gathered}
\end{equation*}
By letting $m\uparrow\infty$, we get
$$E\Big[\sum_{n=-1}^{\infty}\mathbbm1_{\tau_n<\tau_\delta}\sup_{\tau_n\le t<\tau_{n+1}}\big||\xi_{t\wedge\tau_\delta}|^2-|\xi_{\tau_n}|^2\big|\Big]\le N\overline{\mathrm B}(x,\xi).$$
It follows that
\begin{align*}
E\sup_{t<\tau_\delta}|\xi_t|^2\le&|\xi|^2+\sum_{n=-1}^\infty E\mathbbm 1_{\tau_n<\tau_\delta}\sup_{\tau_n\le t<\tau_{n+1}}\big||\xi_{t\wedge\tau_\delta}|^2-|\xi_{\tau_n}|^2\big|\le|\xi|^2+N\overline{\mathrm B}(x,\xi).
\end{align*}
We can prove (2)-(5) by repeating the argument above, and (6) is implied by (5).

The inequality in (7) can also be proved very similarly. To be precise, we start from observing that
\begin{align*}
&E\Big[\mathbbm1_{\tau_{n+1}<\gamma}\B_{n+1}(x_{\tau_{n+1}},\xi_{\tau_{n+1}})+\mathbbm1_{\tau_n<\gamma}\big|\underline\B(x_{\tau_{n+1}\wedge\gamma},\xi_{\tau_{n+1}\wedge\gamma})-\underline\B(x_{\tau_{n}},\xi_{\tau_{n}})\big|\Big]\\
=&E\Big\{E\Big[\mathbbm1_{\tau_{n+1}<\gamma}\B_{n+1}(x_{\tau_{n+1}},\xi_{\tau_{n+1}})+\mathbbm1_{\tau_n<\gamma}\big|\underline\B(x_{\tau_{n+1}\wedge\gamma},\xi_{\tau_{n+1}\wedge\gamma})-\underline\B(x_{\tau_{n}},\xi_{\tau_{n}})\big|\Big|\mathcal F_{\tau_n}\Big]\Big\}\\
\le&E\big[\mathbbm1_{\tau_{n+1}<\gamma}\B_{n+1}(x_{\tau_{n}},\xi_{\tau_{n}})+\mathbbm1_{\tau_n<\gamma}2\B_{n+1}(x_{\tau_{n}},\xi_{\tau_{n}})\big]\\
\le&E\mathbbm1_{\tau_n<\gamma}3\B_{n+1}(x_{\tau_n},\xi_{\tau_n})\\
\le& E\mathbbm1_{\tau_n<\gamma}\B_{n}(x_{\tau_n},\xi_{\tau_n}).
\end{align*}
Then a similar argument leads to
$$E^\alpha_{x,\xi}\underline{\mathrm B}(x_\gamma,\xi_\gamma)\le \underline{\mathrm B}(x,\xi)+\overline{\mathrm B}(x,\xi)\le 2\overline{\mathrm B}(x,\xi).$$

Lemmas \ref{thm1}, \ref{thm2} and \ref{markovtime} are true since the assumptions in Lemma \ref{meaningful} hold for each $t\in[0,\tau^{\alpha,x}_\delta)$.
\end{proof}

\section{Proof of Theorem {\ref{t2}}}\label{sub5}

In the proof, for the simplicity of the notation, we may drop the superscripts such as $\alpha$ and $x$ when this will cause no confusion. 

\begin{proof}[Proof of (\ref{e1})] 

First, we fix $x\in D_\delta$ and $\xi\in\Rd$. Choose a sufficiently small positive $ \epsilon_0$, such that $B(x,\epsilon_0|\xi|):=\{y:|y-x|\le \epsilon_0|\xi|\}\subset D_\delta$. For any $ \epsilon\in(0, \epsilon_0)$, by Assumption \ref{a2}, 
$$v(x+ \epsilon\xi)=\sup_{(\alpha,\beta)\in \mathfrak{A}\times \mathfrak B}E^{\alpha,\beta}_{x+ \epsilon\xi}\Big[g(\tilde x_{\tau})e^{-\phi_{\tau}}+\int_0^{\tau}f^{\alpha_s}(\tilde x_s)e^{-\phi_s}ds\Big],$$
where $\mathfrak B$ is the set of all progressively-measurable processes $\beta$ with value in $\mathbb O^d$ for all $t\ge0$ and
$$\tilde x_t^{\alpha,x}=x+\int_0^t\beta_s\sigma^{\alpha_s}dw_s+\int_0^t b^{\alpha_s}ds.$$
By Bellman's principle (Theorem 1.1 in \cite{MR2303952}, in which $Q$ is defined by $D\times [-1,T+1]$, where $T$ is an arbitrary positive constant), we have, with the stopping time $\tau\le\tau_{\delta}^{\alpha,\beta,x+ \epsilon\xi}\wedge T$,
$$
v(x+ \epsilon\xi)=\sup_{(\alpha,\beta)\in \mathfrak{A}\times\mathfrak B}E^{\alpha,\beta}_{x+ \epsilon\xi}\Big[v(\tilde x_{\tau})e^{-\phi_{\tau}}+\int_0^{\tau}f^{\alpha_s}(\tilde x_s)e^{-\phi_s}ds\Big]
$$
Applying Theorem 2.1 in \cite{MR631436} and Lemmas 2.1 and 2.2 in \cite{MR637615}, we have, with the stopping time $\tau\le\bar\tau_{\delta}^{\alpha,\beta,x+ \epsilon\xi}\wedge T$,
\begin{align}\label{vxplusex}
v(x+\epsilon\xi)=&\sup_{(\alpha,\beta)\in\mathfrak{A}\times\mathfrak B}E^{\alpha,\beta}_{x+\epsilon\xi}\bigg[v(\tilde y_{\tau}( \epsilon))p_{\tau}( \epsilon) e^{-{\phi}_{\tau}( \epsilon)}+\tilde q_\tau(\epsilon)\bigg],
\end{align}
where $\tilde y_t^{\alpha,\beta,y}( \epsilon)$ is defined by (\ref{itoy}) with $\exp(\epsilon P_t^\alpha)$ replaced by $\beta_t$ and
\begin{align}
\phi_t^{\alpha}( \epsilon)=&\int_0^t\theta_s^\alpha(\epsilon)c^{\alpha}ds,\\
\label{pt}
p_t^\alpha( \epsilon)=&\exp\bigg(\int_0^t \epsilon\pi^{\alpha}_sdw_s-\frac{1}{2}\int_0^t| \epsilon\pi^{\alpha}_s|^2ds\bigg),\\
\tilde q_t^{\alpha,\beta,y}(\epsilon)=&\int_0^{t}\theta_s^\alpha(\epsilon)f^{\alpha_s}(\tilde y_s^{\alpha,\beta,y}( \epsilon))p_s^\alpha( \epsilon)e^{-\phi_s^{\alpha}( \epsilon)}ds,\label{36}
\end{align}
with $r_s^{\alpha}, \pi_s^{\alpha}, P_s^{\alpha}$ defined in Lemma \ref{applicable}. Observe that $(\exp(\epsilon P_t^\alpha))_{t\ge0}\in\mathfrak B$, which implies that
$$\{y^{\alpha,y}_t:\alpha\in\mathfrak A\}\subset\{\tilde y^{\alpha,\beta,y}_t:\alpha\in\mathfrak A,\beta\in\mathfrak B\}.$$
Consequently, from (\ref{vxplusex}) we have
\begin{equation}
v(x+\epsilon\xi)\ge\sup_{\alpha\in\mathfrak A}E_{x+\epsilon\xi}^\alpha\bigg[v(y_{\tau}( \epsilon))p_{\tau}( \epsilon) e^{-{\phi}_{\tau}( \epsilon)}+q_\tau(\epsilon)\bigg],
\end{equation}
where $y_t^{\alpha,y}(\epsilon)$ is defined by (\ref{itoy}) and $q_t^{\alpha,y}(\epsilon)$ is defined by (\ref{36}) in which $\tilde y_t^{\alpha,\beta,y}$ is substituted with $y_t^{\alpha,y}$, i.e.
\begin{equation}\label{37}
q_t^{\alpha,y}(\epsilon)=\int_0^{t}\theta_s^\alpha(\epsilon)f^{\alpha_s}(y_s^{\alpha,y}( \epsilon))p_s^\alpha( \epsilon)e^{-\phi_s^{\alpha}( \epsilon)}ds
\end{equation}

To make the expression shorter, for any $\bar x=(x,x^{d+1},x^{d+2},x^{d+3})\in D\times[0,\infty)\times[0,\infty)\times\mathbb R$, we introduce
\begin{equation}\label{V}
V(\bar x)=v(x)\exp(-x^{d+1})x^{d+2}+x^{d+3}.
\end{equation}
If we also define
\begin{align*}
&\bar y_t^{\alpha,y}(\epsilon)=(y_t^{\alpha,y}(\epsilon), \phi_t^{\alpha}(\epsilon),p_t^\alpha(\epsilon),q_t^{\alpha,y}(\epsilon)),\qquad \bar x_t^{\alpha,x}=\bar y_t^{\alpha,x}(0),
\end{align*}
then for the stopping times 
$$\gamma^\alpha=\gamma^\alpha(\epsilon, T, n)=:\bar\tau_\delta^{\alpha,x+\epsilon\xi}\wedge\tau_\delta^{\alpha,x}\wedge T\wedge\vartheta^{\alpha,\xi}_n,$$
where $T\in [1,\infty)$ is constant, and $\vartheta_n^{\alpha,\xi}=\tau_\delta^{\alpha,x}\wedge\inf\{t\ge0:|\xi^{\alpha,\xi}_t|\ge n\}$, we have
$$v(x+ \epsilon\xi)\ge\sup_{\alpha\in\mathfrak{A}}E^\alpha_{x+\epsilon\xi}V(\bar y_\gamma(\epsilon)),\qquad v(x)=\sup_{\alpha\in\mathfrak{A}}E^\alpha_xV(\bar x_\gamma).$$

Due to the inequality $|\sup_\alpha \mathfrak f^\alpha-\sup_\alpha \mathfrak g^\alpha|\le\sup_\alpha|\mathfrak f^\alpha-\mathfrak g^\alpha|$, we have
\begin{align}
-\frac{v(x+ \epsilon\xi)-v(x)}{ \epsilon}\le&\sup_{\alpha\in\mathfrak{A}}E\bigg|\frac{V(\bar y^{\alpha,x+\epsilon\xi}_{\gamma^\alpha}(\epsilon))-V(\bar x^{\alpha,x}_{\gamma^\alpha})}{ \epsilon}\bigg|\nonumber\\
\le& I_1(\epsilon,T,n)+I_2(\epsilon,T,n).\label{i1i2}
\end{align}
Here
\begin{align*}
I_1(\epsilon, T,n)=&\sup_{\alpha\in\mathfrak{A}}E\bigg|\frac{V(\bar y^{\alpha,x+\epsilon\xi}_{\gamma^\alpha}(\epsilon))-V(\bar x^{\alpha,x}_{\gamma^\alpha})}{ \epsilon}-V_{(\bar\xi_{\gamma^\alpha}^{\alpha,\xi})}(\bar x_{\gamma^\alpha}^{\alpha,x})\bigg|,\\
I_2(\epsilon, T,n)=&\sup_{\alpha\in\mathfrak{A}}E|V_{(\bar\xi_{\gamma^\alpha}^{\alpha,\xi})}(\bar x_{\gamma^\alpha}^{\alpha,x})|,
\end{align*}
where
\begin{equation}\label{xibar}
\bar\xi_t^{\alpha,\xi}=(\xi_t^{\alpha,\xi},\xi_t^{d+1,\alpha},\xi_t^{d+2,\alpha},\xi_t^{d+3,\alpha}),
\end{equation}
with $\xi_t^{\alpha,\xi}$ the solution to the It\^o stochastic equation (\ref{itoxi}) and
\begin{align}
\xi_t^{d+1,\alpha}=&\int_0^t2r^\alpha_sc^{\alpha_s}ds,\\
\xi_t^{d+2,\alpha}=&\int_0^t\pi^\alpha_sdw_s\Big(=\tilde\xi^{\alpha,0}_t\Big),\\
\xi_t^{d+3,\alpha}=&\int_0^te^{-\phi_s^{\alpha}}\Big[f^{\alpha_s}_{(\xi_s^{\alpha,\xi})}(x_s^{\alpha,x})\label{42}\\
&+\big(2r^\alpha_s-\xi_s^{d+1,\alpha}+\xi_s^{d+2,\alpha}\big)f^{\alpha_s}(x_s^{\alpha,x})\Big]ds.\nonumber
\end{align}

We claim that 
\begin{equation}\label{i1}
\lim_{\epsilon\downarrow0}I_1(\epsilon,T,n)=0.
\end{equation}
By Lemma \ref{thm1}, it suffices to prove that
\begin{equation}
\lim_{\epsilon\downarrow0}\bigg(\sup_{\alpha\in\mathfrak A}E\sup_{t\le\gamma}\Big|\frac{\bar y_t^{\alpha,x+\epsilon\xi}(\epsilon)-\bar x_t^{\alpha,x}}{\epsilon}-\bar\xi_t^{\alpha,\xi}\Big|\bigg)=0. \label{barlimy}
\end{equation}
In other words, we just need to show
\begin{align}
&\lim_{\epsilon\downarrow0}\bigg(\sup_{\alpha\in\mathfrak A}E\sup_{t\le\gamma}\Big|\frac{ y_t^{\alpha,x+\epsilon\xi}(\epsilon)- x_t^{\alpha,x}}{\epsilon}-\xi_t^{\alpha,\xi}\Big|\bigg)=0,\label{lim1}\\
&\lim_{\epsilon\downarrow0}\bigg(\sup_{\alpha\in\mathfrak A}E\sup_{t\le\gamma}\Big|\frac{ \phi_t^{\alpha}(\epsilon)- \phi_t^{\alpha}}{\epsilon}-\xi_t^{d+1,\alpha}\Big|\bigg)=0,\label{lim2}\\
&\lim_{\epsilon\downarrow0}\bigg(\sup_{\alpha\in\mathfrak A}E\sup_{t\le\gamma}\Big|\frac{ p_t^{\alpha}(\epsilon)- 1}{\epsilon}-\xi_t^{d+2,\alpha}\Big|\bigg)=0,\label{lim3}\\
&\lim_{\epsilon\downarrow0}\bigg(\sup_{\alpha\in\mathfrak A}E\sup_{t\le\gamma}\Big|\frac{ q_t^{\alpha,x+\epsilon\xi}(\epsilon)- q_t^{\alpha,x}}{\epsilon}-\xi_t^{d+3,\alpha}\Big|\bigg)=0\label{lim4}.
\end{align}
The equation (\ref{lim1}) is exactly (\ref{res1b2}) with $p=1$, which has already been verified. The equation (\ref{lim2}) is true because of (\ref{guji}).
To prove (\ref{lim3}), we notice that
$$\frac{p_t^\alpha(\epsilon)-1}{\epsilon}-\xi_t^{d+2,\alpha}=\int_0^t\big(p^\alpha_s(\epsilon)-1\big)\pi^\alpha_sdw_s.$$
Recall that the stopping time $\gamma^\alpha$ is bounded by $T\wedge\vartheta_n^{\alpha,\xi}$. It follows by Davis inequality that
\begin{align*}
E^\alpha\sup_{t\le\gamma\wedge\gamma_m}\Big|\frac{ p_t(\epsilon)- 1}{\epsilon}-\xi_t^{d+2}\Big|
\le&3E^\alpha\Big(\int_0^{\gamma\wedge\gamma_m}\big(p_t(\epsilon)-1\big)^2|\pi_t|^2dt\Big)^{1/2}\\
\le&3\epsilon E^\alpha\bigg[\sup_{t\le\gamma\wedge\gamma_m}\Big|\frac{p_t(\epsilon)-1}{\epsilon}\Big|\Big(\int_0^{\gamma\wedge\gamma_m}|\pi_t|^2dt\Big)^{1/2}\bigg]\\
\le&3\epsilon\sqrt Tn N/\delta E^\alpha\sup_{t\le\gamma\wedge\gamma_m}\Big|\frac{p_t(\epsilon)-1}{\epsilon}\Big|\\
\le&3\epsilon\sqrt T nN/\delta E^\alpha\sup_{t\le\gamma\wedge\gamma_m}\bigg(\Big|\frac{p_t(\epsilon)-1}{\epsilon}-\xi_t^{d+2}\Big|+|\xi_t^{d+2}|\bigg).
\end{align*}
where $\gamma_m$ is a localizing sequence of stopping times such that the left hand side of the inequalities is finite for each $m$. Collecting similar terms to the left side of the inequality and then letting $m\rightarrow\infty$, by the monotone convergence theorem, we obtain
\begin{align*}
\big(1-3\epsilon\sqrt TnN/\delta\big)E^\alpha\sup_{t\le\gamma}\Big|\frac{ p_t(\epsilon)- 1}{\epsilon}-\xi_t^{d+2}\Big|\le&3\epsilon\sqrt T nN/\delta E^\alpha\Big(\int_0^\gamma|\pi_t|^2dt\Big)^{1/2}.
\end{align*}
Hence (\ref{lim3}) is obtained by first taking the supremum over $\mathfrak A$ and then letting $\epsilon\downarrow0$.

To prove (\ref{lim4}), for each $\alpha\in A$, we introduce the  function:
\begin{equation}\label{F}
F^\alpha: \bar D\times [0,\infty)\times[0,\infty)\times\mathbb R\rightarrow \mathbb R;\  \bar x\mapsto f^\alpha(x)\exp(-x^{d+1})x^{d+2}. 
\end{equation}
From (\ref{37}) and (\ref{42}) we have
\begin{align*}
\frac{ q_t^{\alpha,x+\epsilon\xi}(\epsilon)- q_t^{\alpha,x}}{\epsilon}-\xi_t^{d+3,\alpha}=&\int_0^t\bigg[\frac{ F^{\alpha_s}(\bar y_s^{\alpha,x+\epsilon\xi}(\epsilon))-F^{\alpha_s}( \bar x_s^{\alpha,x})}{\epsilon}-F^{\alpha_s}_{(\bar\xi_s^{\alpha,\xi})}(\bar x_s^{\alpha,x})\\
&+\frac{\arctan(\pi2\epsilon r_s^\alpha)}{\epsilon\pi}F^{\alpha-s}(\bar y_s^{\alpha,x+\epsilon\xi}(\epsilon))-2r_s^\alpha F^{\alpha_s}(\bar x^{\alpha,x}_s)\bigg]ds
\end{align*}
To prove (\ref{lim4}) it suffices to show that
\begin{equation}\label{lim5}
\lim_{\epsilon\downarrow0}\bigg(\sup_{\alpha\in\mathfrak A}E\sup_{t\le\gamma}\Big|\frac{ F^{\alpha_t}(\bar y_t^{\alpha,x+\epsilon\xi}(\epsilon))-F^{\alpha_t}( \bar x_t^{\alpha,x})}{\epsilon}-F^{\alpha_t}_{(\bar\xi_t^{\alpha,\xi})}(\bar x_t^{\alpha,x})\Big|\bigg)=0,
\end{equation}
\begin{equation}\label{lim6}
\lim_{\epsilon\downarrow0}\bigg(\sup_{\alpha\in\mathfrak A}E\sup_{t\le\gamma}\Big|\frac{\arctan(\pi2\epsilon r_t^\alpha)}{\epsilon\pi}F^{\alpha_t}(\bar y_t^{\alpha,x+\epsilon\xi}(\epsilon))-2r_t^\alpha F^{\alpha_t}(\bar x^{\alpha,x}_t)\Big|\bigg)=0.
\end{equation}
which are valid due to (\ref{lim1})-(\ref{lim3}) and (\ref{guji}). Therefore (\ref{lim4}) is proved.

We have obtained (\ref{i1}). Next, we estimate $I_2(\epsilon, T,n)$. From (\ref{V}) we have
$$V_{(\bar\xi^{\alpha,\xi}_t)}(\bar x^{\alpha,x}_t)=e^{-\phi_t^{\alpha,x}}v_{(\xi_t^{\alpha,\xi})}(x_t^{\alpha,x})+X_t^{\alpha,x,\xi},$$
where
\begin{align*}
X_t^{\alpha,x,\xi}=&e^{-\phi_t^{\alpha,x}}(\xi^{d+2,\alpha}_t-\xi_t^{d+1,\alpha})v(x_t^{\alpha,x})+\xi_t^{d+3,\alpha}.
\end{align*}
It follows that
$$I_2(\epsilon, T,n)=\sup_{\alpha\in\mathfrak A}E^\alpha_{x,\xi}|V_{(\bar\xi_\gamma)}(\bar x_\gamma)| \le \sup_{\alpha\in\mathfrak A}E^\alpha_{x,\xi} |v_{(\xi_{\gamma})}(x_{\gamma})|+\sup_{\alpha\in\mathfrak A}E^\alpha_{x,\xi}|X_{\gamma}|.$$

We first claim that
\begin{equation}\label{claimx}
\sup_{\alpha\in\mathfrak A}E^\alpha_{x,\xi}|X_{\gamma}|\le N\overline{\mathrm B}^{1/2}(x,\xi),
\end{equation}
where $N$ is independent of $\epsilon$, $T$, $n$.
Indeed, from the definition of $X_t$,
\begin{equation*}
|X_\gamma|\le|v|_{0,D}\big(|\xi_\gamma^{d+1}|+|\xi_\gamma^{d+2}|\big)\\
+|f|_{1,D}\bigg[\int_0^\gamma e^{-\phi_t}\Big(|\xi_t|+|\xi_t^{d+1}|+|\xi_t^{d+2}|+2|r_t|\Big)dt\bigg]
\end{equation*}
Notice that we have the following estimates:
\begin{align*}\allowdisplaybreaks
|v|_{0,D}\le& |g|_{0,D}+|\psi|_{0,D}\sup_{\alpha\in A}|f^\alpha|_{0,D}\\
\sup_{\alpha\in\mathfrak A}E^\alpha|\xi_\gamma^{d+1}|\le&K\sup_{\alpha\in\mathfrak A}E^\alpha\int_0^\gamma\Big(|\xi_t|+\frac{|\psi_{(\xi_t)}|}{\psi}\Big)dt\\
\le& N\overline{\mathrm B}^{1/2}(x,\xi)\\
\sup_{\alpha\in\mathfrak A}E^\alpha|\xi_\gamma^{d+2}|\le&3\sup_{\alpha\in\mathfrak A}E^\alpha\langle\xi^{d+2}\rangle_\gamma^{1/2}\\
\le& N\overline{\mathrm B}^{1/2}(x,\xi)\\
\sup_{\alpha\in\mathfrak A}E^\alpha\int_0^\gamma\big(|\xi_t|+|r_t|\big)dt\le&K\sup_{\alpha\in\mathfrak A}E^\alpha\int_0^\gamma\Big(|\xi_t|+\frac{|\psi_{(\xi_t)}|}{\psi}\Big)dt\\
\le& N\overline{\mathrm B}^{1/2}(x,\xi)\\
\sup_{\alpha\in\mathfrak A}E^\alpha\int_0^\gamma e^{-\phi_t}|\xi_t^{d+1}|dt\le&\sfa E^\alpha\int_0^\gamma2|r_t|dt\int_0^\gamma e^{-ct}cdt\\
\le&K\sup_{\alpha\in\mathfrak A}E^\alpha\int_0^\gamma\Big(|\xi_t|+\frac{|\psi_{(\xi_t)}|}{\psi}\Big)dt\\
\le& N\overline{\mathrm B}^{1/2}(x,\xi)\\
\sup_{\alpha\in\mathfrak A}E^\alpha\int_0^\gamma |\xi_t^{d+2}|dt\le&\sfa E^\alpha\gamma\sup_{t\le\gamma}|\xi_t^{d+2}|\\
\le&\sfa \big(E^\alpha \gamma^2\big)^{1/2} \big(E^\alpha \sup_{t\le\gamma}|\xi_t^{d+2}|^2\big)^{1/2}\\
\le&N\overline{\mathrm B}^{1/2}(x,\xi)
\end{align*}
Applying the estimates above, (\ref{claimx}) is proved.

We also claim that
\begin{equation}\label{250}
\begin{gathered}
\varlimsup_{n\uparrow\infty}\varlimsup_{T\uparrow\infty}\varlimsup_{\epsilon\downarrow0}\sup_{\alpha\in\mathfrak A}E^\alpha_{x,\xi} |v_{(\xi_{\gamma})}(x_{\gamma})|\\
\le\bigg(\sup_{\substack{y\in\{\psi=\delta\}\\\zeta\in\Rd\setminus\{0\}}}\frac{|v_{(\zeta)}(y)|}{\sqrt{\mathrm B_1(y,\zeta)}}+2\bigg)\sqrt{2\overline{\mathrm B}(x,\xi)}.
\end{gathered}
\end{equation}
Indeed, we notice that
\begin{align*}
\sup_{\alpha\in\mathfrak{A}}E^\alpha_{x,\xi}|v_{(\xi_{\gamma})}(x_{\gamma})|=&\sup_{\alpha\in\mathfrak{A}}E^\alpha_{x,\xi}\frac{|v_{(\xi_{\gamma})}(x_{\gamma})|}{\sqrt{\underline{\mathrm{B}}(x_{\gamma},\xi_{\gamma})}}\cdot\sqrt{\underline{\mathrm{B}}(x_{\gamma},\xi_{\gamma})}\\
\le&J_1(\epsilon, T,n)+J_2(\epsilon, T,n),
\end{align*}
where
\begin{align*}
J_1(\epsilon, T,n)=&\sup_{\alpha\in\mathfrak{A}}E^\alpha_{x,\xi}\bigg(\frac{|v_{(\xi_{\gamma})}(x_{\gamma})|}{\sqrt{\underline{\mathrm{B}}(x_{\gamma},\xi_{\gamma})}}-\frac{|v_{(\xi_{\gamma})}(x_{\tau_\delta})|}{\sqrt{\underline{\mathrm{B}}(x_{\tau_\delta},\xi_{\gamma})}}\bigg)\sqrt{\underline{\mathrm{B}}(x_{\gamma},\xi_{\gamma})},\\
J_2(\epsilon, T,n)=&\sup_{\alpha\in\mathfrak{A}}E^\alpha_{x,\xi}\frac{|v_{(\xi_{\gamma})}(x_{\tau_\delta})|}{\sqrt{\underline{\mathrm{B}}(x_{\tau_\delta},\xi_{\gamma})}}\sqrt{\underline{\mathrm{B}}(x_{\gamma},\xi_{\gamma})}.
\end{align*}
Note that
$$v_{(\xi)}(x)/\sqrt{\underline{\mathrm{B}}(x,\xi)}=v_{(\xi/|\xi|)}(x)/\sqrt{\underline{\mathrm{B}}(x,\xi/|\xi|)}$$
is a continuous function from $\bar D_\delta\times S_1$ to $\mathbb{R}$, where $S_1$ is the unit sphere in $\Rd$. By Weierstrass approximation theorem,  there exists a polynomial $W(x,\xi): \bar D_\delta\times S_1\rightarrow\mathbb{R}$, such that 
$$\sup_{x\in D_\delta,\xi\in S_1}\Big|\frac{v_{(\xi)}(x)}{\sqrt{\underline{\mathrm{B}}(x,\xi)}}-W(x,\xi)\Big|\le1.$$
It follows that
\begin{align*}
J_1(\epsilon,T,n)\le&\sup_{\alpha\in\mathfrak{A}}E^\alpha_{x,\xi}\big|W(x_{\gamma},\xi_{\gamma}/|\xi_\gamma|)-W(x_{\tau_\delta},\xi_{\gamma}/|\xi_\gamma|)\big|\sqrt{\underline{\mathrm{B}}(x_{\gamma},\xi_{\gamma})}\\
&+2\sup_{\alpha\in\mathfrak{A}}E^\alpha_{x,\xi}\sqrt{\underline{\mathrm{B}}(x_{\gamma},\xi_{\gamma})}\\
\le&(N/\delta)\sup_{\alpha\in\mathfrak{A}}E^\alpha_{x,\xi}|x_{\gamma}-x_{\tau_\delta}||\xi_\gamma|(\mathbbm1_{\tau_\delta\le\vartheta_n}+\mathbbm1_{\tau_\delta>\vartheta_n})+2\sqrt{2\overline{\B}(x,\xi)}\\
\le&(Nn/\delta)\sup_{\alpha\in\mathfrak{A}}E^\alpha_{x}\Big[(\tau_\delta-\gamma)+\sqrt{\tau_\delta-\gamma}\Big]+(N/\delta)\sup_{\alpha\in\mathfrak{A}}E^\alpha_{x,\xi}|\xi_\gamma|\mathbbm1_{\tau_\delta>\vartheta_n}\\
&+2\sqrt{2\overline{\B}(x,\xi)}.
\end{align*}
Notice that
\begin{align*}
E_x^\alpha(\tau_\delta-\gamma)\le&E(\tau_\delta^{\alpha,x}-\tau_\delta^{\alpha,x}\wedge\tau_\delta^{\alpha,x+\epsilon\xi})+E(\tau_\delta^{\alpha,x}-\tau_\delta^{\alpha,x}\wedge T),\\
E^\alpha_{\xi}|\xi_\gamma|\mathbbm1_{\tau_\delta>\vartheta_n}\le&\sqrt{E^\alpha_{x,\xi}|\xi_\gamma|^2}\sqrt{P^\alpha_{x,\xi}\Big(\sup_{t\le\tau_\delta}|\xi_t|\ge n\Big)}\le (1/n)E^\alpha_\xi\sup_{t\le\tau_\delta}|\xi_t|^2.
\end{align*}
Thus by Lemma \ref{markovtime} and Lemma \ref{applicable} (1),
$$\varlimsup_{n\uparrow\infty}\varlimsup_{T\uparrow\infty}\varlimsup_{\epsilon\downarrow0}J_1(\epsilon,T,n)\le2\sqrt{2\overline{\B}(x,\xi)}.$$
Also, notice that
\begin{align*}
J_2(\epsilon,T,n)\le&\sup_{\substack{y\in\{\psi=\delta\}\\\zeta\in  \Rd\setminus\{0\}}}\frac{|v_{(\zeta)}(y)|}{\sqrt{\mathrm{B}_1(y,\zeta)}}\cdot\sqrt{2\overline{\B}(x,\xi)}.
\end{align*}
Thus (\ref{250}) is proved.

Combining (\ref{claimx}) and (\ref{250}), we obtain
\begin{equation}\label{i2}
\varlimsup_{n\uparrow\infty}\varlimsup_{T\uparrow\infty}\varlimsup_{\epsilon\downarrow0}I_2(\epsilon,T,n)\le\bigg(\sup_{\substack{y\in \{\psi=\delta\}\\\zeta\in  \Rd\setminus\{0\}}}\frac{|v_{(\zeta)}(y)|}{\sqrt{\mathrm{B}_1(y,\zeta)}}+N\bigg)\sqrt{\overline{\B}(x,\xi)}.
\end{equation}

It remains to estimate
$$
\varlimsup_{\delta\downarrow0}\bigg(\sup_{\substack{x\in\{\psi=\delta\}\\\xi\in\Rd\setminus\{0\}}}\vxosbw\bigg).
$$
Due to the compactness of $(\partial D_\delta)\times S_1$, for each $\delta$, there exist $x(\delta)\in\partial D_\delta$ and $\xi(\delta)\in S_1$, such that
$$
\sup_{\substack{x\in\{\psi=\delta\}\\\xi\in\Rd\setminus\{0\}}}\vxosbw=\frac{|v_{(\xi(\delta))}(x(\delta))|}{\sqrt{\mathrm{B}_1(x(\delta),\xi(\delta))}}.
$$
A subsequence of $(x(\delta),\xi(\delta))$ converges to some $(y,\zeta)$, where $y\in\partial D$ and $|\zeta|=1$.

If $\psi_{(\zeta)}(y)\ne0$, then $\mathrm{B}_1(x(\delta),\xi(\delta))\rightarrow\infty$ as $\delta\downarrow0$. In this case,
$$
\varlimsup_{\delta\downarrow0}\bigg(\sup_{\substack{x\in\{\psi=\delta\}\\\xi\in\Rd\setminus\{0\}}}\vxosbw\bigg)=\varlimsup_{\delta\downarrow0}\frac{|v_{(\xi(\delta))}(x(\delta))|}{\sqrt{\mathrm{B}_1(x(\delta),\xi(\delta))}}=0.
$$

If $\psi_{(\zeta)}(y)=0$, then $\zeta$ is tangent to $\partial D$ at $y$. In this case,
\begin{equation*}\label{lims}
\varlimsup_{\delta\downarrow0}\bigg(\sup_{\substack{x\in\{\psi=\delta\}\\\xi\in\Rd\setminus\{0\}}}\vxosbw\bigg)=\varlimsup_{\delta\downarrow0}\frac{|v_{(\xi(\delta))}(x(\delta))|}{\sqrt{\mathrm{B}_1(x(\delta),\xi(\delta))}}=\frac{|g_{(\zeta)}(y)|}{\lambda}.
\end{equation*}
Therefore for all $x\in D$ and $\xi\in\Rd$, we have
\begin{align*}
-v_{(\xi)}(x)\le N\sqrt{\mathbbm1_{x\in D_\delta^\lambda}B_1(x,\xi)+\mathbbm1_{x\in D_{\lambda^2}}B_2(x,\xi)}.
\end{align*}

Substituting $\xi$ with $-\xi$ completes the proof of the inequality (\ref{e1}).
\end{proof}

\section{Proof of Theorem {\ref{t3}}}\label{sub6}

To estimate the second derivatives of $v$, we don't need to take effort on making the second quasiderivatives tangent to the boundary when the state process exits the domain. This is due to the following lemma.

\begin{lemma}\label{lemma4}
If $f^\alpha, g\in C^2(\bar D)$, and $v\in C^1(\bar D)$, then for any $y\in\partial D$,
\begin{equation}
|v_{(n)}(y)|\le K\Big(|g|_{2,D}+\sup_{\alpha\in A}|f^\alpha|_{0,D}\Big)\label{normal},
\end{equation}
where $n$ is the unit inward normal on $\partial D$ and the constant $K$ depends only on $K_0, d, d_0$ and $D$.
\end{lemma}

\begin{proof}
Fix a $y\in\partial D$, and choose $\ve_0>0$ so that $y+\ve n\in D$ as long as $0<\ve\le\ve_0$. Let $x:=y+\ve n$. For any $\alpha\in \mathfrak A$, 
\begin{equation}\label{normal1}
E^\alpha_x\big[g(x_\tau)e^{-\phi_\tau}\big]=g(x)+E^\alpha_x\bigg[\int_0^\tau\big(Lg(x_t)-c(x_t)g(x_t)\big)e^{-\phi_t}dt\bigg]
\end{equation}
From (\ref{v}), we have
\begin{equation}\label{normal2}
v(x)\ge v^\alpha(x)=E^\alpha_x\bigg[g(x_\tau)e^{-\phi_\tau}+\int_0^\tau f^\alpha(x_t)e^{-\phi_t}dt\bigg]
\end{equation}
Combining (\ref{normal1}) and (\ref{normal2}),
\begin{align}
v(x)\ge&g(x)+E_x^\alpha\int_0^\tau e^{-\phi_t}\big[(L-c)g(x_t)+f(x_t)\big]dt\nonumber\\
\ge&g(x)-(|Lg|_{0,D}+|c|_{0,D}|g|_{0,D}+|f^\alpha|_{0,D})E_x^\alpha\tau\label{50}\\
\ge&g(x)-K(|g|_{2,D}+\sup_{\alpha\in A}|f^\alpha|_{0,D})\psi(x)\nonumber.
\end{align}
Notice that $u(y)=g(y)$ and $\psi(y)=0$, so we have
\begin{align*}
\frac{v(y+\epsilon n)-v(y)}{\epsilon}\ge&\frac{g(y+\epsilon n)-g(y)}{\epsilon}-K(|g|_{2,D}+\sup_{\alpha\in A}|f^\alpha|_{0,D})\frac{\psi(y+\epsilon n)-\psi(y)}{\epsilon}.
\end{align*}
Letting $\ve\downarrow0$, we get
$$v_{(n)}(y)\ge -K(|g|_{2,D}+\sup_{\alpha\in A}|f^\alpha|_{0,D}).$$
To estimate $v_{(n)}(y)$ from above we first notice that for any $\theta>0$, there exists $\alpha(\theta)\in\mathfrak A$, such that
$$v(x)\le v^{\alpha(\theta)}(x)+\theta.$$
A sequence of inequalities similar to (\ref{50}) implies that
$$v(x)\le g(x)+K(|g|_{2,D}+\sup_{\alpha\in A}|f^\alpha|_{0,D})\psi(x)+\theta$$
It remains to let $\theta\downarrow 0$ and then mimic the argument after (\ref{50}).

\end{proof}

\begin{proof}[Proof of the first inequality in (\ref{e3})]

The idea is similar to that in the first order case. Fix $x\in D_\delta$, $\xi\in\Rd$ and a sufficiently small positive $ \epsilon_0$, so that $B(x,\epsilon_0|\xi|)\subset D_\delta$. Repeating the argument of deriving (\ref{vxplusex}), we obtain a similar representation of a lower bound of $v(x+\epsilon\xi)$, namely, for any stopping time $\tau$ satisfying $\tau\le\tau_{\delta}^{\alpha,x+ \epsilon\xi}\wedge T$,
\begin{equation*}
v(x+\epsilon\xi)\ge\sup_{\alpha\in\mathfrak{A}}E^\alpha_{x+\epsilon\xi}\bigg[v(z_{\tau}( \epsilon))\hat p_{\tau}( \epsilon) e^{-{\hat \phi}_{\tau}( \epsilon)}+\hat q_\tau(\epsilon)\bigg],
\end{equation*}
where $z_t^{\alpha,z}(\epsilon)$ is defined by (\ref{itoz}) and
\begin{align}
&\hat\phi_t^{\alpha}( \epsilon)=\int_0^t\hat\theta_t^\alpha(\epsilon)c^{\alpha}ds,\\
&\hat p_t^\alpha( \epsilon)=\exp\bigg(\int_0^t \Big(\epsilon\pi^{\alpha}_s+\frac{\epsilon^2\hat\pi^\alpha_s}{2}\Big)dw_s-\frac{1}{2}\int_0^t\Big| \epsilon\pi^{\alpha}_s+\frac{\epsilon^2\hat \pi^\alpha_s}{2}\Big|^2ds\bigg),\\
&\hat q_t^{\alpha,z}(\epsilon)=\int_0^{t}\hat\theta_t^\alpha(\epsilon)f^{\alpha_s}(z^{\alpha,z}_s( \epsilon))\hat p_s( \epsilon)e^{-\hat\phi^\alpha_s( \epsilon)}ds,\label{hatq}
\end{align}
with $r_s^{\alpha}, \pi_s^{\alpha}, P_s^{\alpha}, \hat r_s^{\alpha}, \hat\pi_s^{\alpha}, \hat P_s^{\alpha}$ defined in Lemma \ref{applicable}. We also define
\begin{align*}
&\bar z_t^{\alpha,z}(\epsilon)=(z_t^{\alpha,z}(\epsilon), \hat\phi_t^{\alpha}(\epsilon),\hat p_t^\alpha(\epsilon),\hat q_t^{\alpha,y}(\epsilon)),\qquad\bar x_t^{\alpha,x}=\bar z_t^{\alpha,x}(0).
\end{align*}
Recall that $V$ is defined by (\ref{V}). Therefore, for the stopping times 
$$\gamma^\alpha=\gamma^\alpha(\epsilon, T, n)=:\hat\tau_{\delta}^{\alpha,x+ \epsilon\xi}\wedge\tau_{\delta}^{\alpha,x}\wedge\hat\tau_{\delta}^{\alpha,x- \epsilon\xi}\wedge T\wedge\vartheta^{\alpha,\xi}_n,$$ 
we have
\begin{align}
&-\frac{v(x+ \epsilon\xi)-2v(x)+v(x- \epsilon\xi)}{ \epsilon^2}\label{g1g2}\\
\le&\frac{1}{\epsilon^2}\Big(-\sup_{\alpha\in\mathfrak A}E^\alpha_{x+\epsilon\xi}V(\bar z_\gamma(\epsilon)+2\sup_{\alpha\in\mathfrak A}E^\alpha_xV(\bar x_\gamma)-\sup_{\alpha\in\mathfrak A}E^\alpha_{x-\epsilon\xi}V(\bar z_\gamma(-\epsilon)\Big)\nonumber\\
\le&\sup_{\alpha\in\mathfrak A}\frac{-E^\alpha_{x+\epsilon\xi}V(\bar z_\gamma(\epsilon)+2E^\alpha_x V(\bar x_\gamma)- E^\alpha_{x-\epsilon\xi}V(\bar z_\gamma(-\epsilon)}{\epsilon^2}\nonumber\\
\le&G_1(\epsilon, T,n)+G_2(\epsilon,T,n).\nonumber
\end{align}
Here
\begin{align*}
G_1(\epsilon,T,n)=&\sup_{\alpha\in\mathfrak A}E\bigg|\frac{-V(\bar z^{\alpha,x+\epsilon\xi}_{\gamma^\alpha}(\epsilon)+2 V(\bar x^{\alpha,x}_{\gamma^\alpha})- V(\bar z^{\alpha,x-\epsilon\xi}_{\gamma^\alpha}(-\epsilon))}{\epsilon^2}\\
&+V_{(\bar \eta^{\alpha,0}_{\gamma^\alpha})}(\bar x^{\alpha,x}_{\gamma^\alpha})+V_{(\bar \xi^{\alpha,\xi}_{\gamma^\alpha})(\bar \xi^{\alpha,\xi}_{\gamma^\alpha})}(\bar x^{\alpha,x}_{\gamma^\alpha})\bigg|,\\
G_2(\epsilon,T,n)=&\sup_{\alpha\in\mathfrak A}E\Big(-V_{(\bar \eta^{\alpha,0}_{\gamma^\alpha})}(\bar x^{\alpha,x}_{\gamma^\alpha})-V_{(\bar \xi^{\alpha,\xi}_{\gamma^\alpha})(\bar \xi^{\alpha,\xi}_{\gamma^\alpha})}(\bar x^{\alpha,x}_{\gamma^\alpha})\Big),
\end{align*}
where $\bar\xi^{\alpha,\xi}_t$ is defined by (\ref{xibar}), and
\begin{align*}
\bar\eta^{\alpha,0}_t:=(\eta_t^{\alpha,0}, \eta^{d+1,\alpha}_t, \eta^{d+2,\alpha}_t, \eta^{d+3,\alpha}_t),
\end{align*}
with $\eta_t^{\alpha,\eta}$ defined by (\ref{itoeta}) and
\begin{align}
\eta^{d+1,\alpha}_t=&\int_0^t2\hat r^\alpha_sc^\alpha ds,\\
\eta^{d+2,\alpha}_t=&\Big(\int_0^t\pi_s^\alpha dw_s\Big)^2-\int_0^t|\pi^\alpha_s|^2ds+\int_0^t\hat\pi_sdw_s\Big(=\tilde\eta^{\alpha,0}_t\Big),\\
\eta^{d+3,\alpha}_t=&\int_0^te^{-\phi_s^{\alpha,x}}\Big\{f^{\alpha_s}_{(\xi_s^{\alpha,\xi})(\xi_s^{\alpha,\xi})}(x_s^{\alpha,x})+f_{(\eta_s^{\alpha,\eta})}(x_s^{\alpha,x})\label{etad3}\\
&+2\big[2r^\alpha_s-\xi_s^{d+1,\alpha}+\xi_s^{d+2,\alpha}\big]f^{\alpha_s}_{(\xi_s^{\alpha,\xi})}(x_s^{\alpha,x})\nonumber\\
&+\big[2\hat r^\alpha_s-4r^\alpha_s(\xi_s^{d+1,\alpha}-\xi_s^{d+2,\alpha})+(\xi_s^{d+1,\alpha})^2\nonumber\\
&-\eta_s^{d+1,\alpha}-2\xi_t^{d+1,\alpha}\xi_t^{d+2,\alpha}+\eta_t^{d+2,\alpha}\big] f^{\alpha_s}(x_s^{\alpha,x})\Big\}ds.\nonumber
\end{align}

By Lemma \ref{thm2}, to obtain
\begin{equation}\label{G1}
\lim_{\epsilon\downarrow0}G_1(\epsilon, T,n)=0,\end{equation}
it suffices to show that
\begin{equation}
\lim_{\epsilon\downarrow0}\bigg(\sup_{\alpha\in\mathfrak A}E\sup_{t\le\gamma}\Big|\frac{\bar z_t^{\alpha,x+\epsilon\xi}(\epsilon)-\bar x_t^{\alpha,x}}{\epsilon}-\bar\xi_t^{\alpha,\xi}\Big|\bigg)=0,\label{barlimz}
\end{equation}
\begin{equation}
\lim_{\epsilon\downarrow0}\bigg(\sup_{\alpha\in\mathfrak A}E\sup_{t\le\gamma}\Big|\frac{\bar z_t^{\alpha,x+\epsilon\xi}(\epsilon)-2\bar x_t^{\alpha,x}+\bar z_t^{\alpha, x-\epsilon\xi}(-\epsilon)}{\epsilon^2}-\bar\eta_t^{\alpha,\xi}\Big|\bigg)=0.\label{barlimzz}
\end{equation}
The convergence result (\ref{barlimz}) can be established by the same way of obtaining (\ref{barlimy}). The convergence result (\ref{barlimzz}) can be proved by showing the same convergence result for each component of the quantity in the absolute value symbol.

The convergence of the first $d$ components is exactly (\ref{res2b3}) which has already been verified. The ($d+1$)-th component is true since
$$\bigg|\frac{\hat \theta_t^{\alpha}(\epsilon)- 2+\hat\theta_t^\alpha(-\epsilon)}{\epsilon^2}-2\hat r_t^{\alpha}\bigg|=\frac{\epsilon}{3}|\hat\theta_t'''(\epsilon')|\le \epsilon C(|r_t^\alpha|^4+|\hat r_t^\alpha|^4).$$
Next, we notice that
\begin{align*}
\frac{\hat p_t(\epsilon)-2\hat p_t(0)+\hat p_t(-\epsilon)}{\epsilon^2}=&\int_0^t\Big(\frac{\hat p_s(\epsilon)-\hat p_s(-\epsilon)}{\epsilon}\pi_s+\frac{\hat p_s(\epsilon)+\hat p_s(-\epsilon)}{2}\hat \pi_s\Big)dw_s,\\
\eta_t^{d+2}=&\int_0^t\big(2\tilde\xi_s\pi_s+\hat\pi_s\big)dw_s.
\end{align*}
It follows that
\begin{align*}
\frac{\hat p_t(\epsilon)-2+\hat p_t(-\epsilon)}{\epsilon^2}-\eta_t^{d+2}=&\int_0^t\mathfrak p_s\pi_sdw_s+\epsilon\int_0^t\mathfrak q_s\hat \pi_sdw_s,
\end{align*}
where
$$\mathfrak p_s=\mathfrak p_s(\epsilon)=\frac{\hat p_s(\epsilon)-\hat p_s(-\epsilon)}{\epsilon}-2\tilde\xi_s, \ \  \mathfrak q_s=\mathfrak q_s(\epsilon)=\frac{\hat p_s(\epsilon)-2+\hat p_s(-\epsilon)}{2\epsilon}.$$
Recall that 
$$\sfa E^\alpha\int_0^\gamma|\hat\pi_t|^2dt\le \sfa E^\alpha\int_0^\gamma|\pi_t|^4dt<\infty.$$
By the triangle inequality, Davis inequality and then H\"older inequality, we have
\allowdisplaybreaks\begin{align*}
&E^\alpha\sup_{t\le\gamma}\Big|\frac{\hat p_t(\epsilon)- 2+\hat p_t(-\epsilon)}{\epsilon^2}-\eta_t^{d+2}\Big|\\
\le&E^\alpha\sup_{t\le\gamma}\Big|\int_0^t\mathfrak p_s\pi_sdw_s\Big|+\epsilon E^\alpha\sup_{t\le\gamma}\Big|\int_0^t\mathfrak q_s\hat \pi_sdw_s\Big|\\
\le&3E^\alpha\Big(\int_0^\gamma\mathfrak p_t^2|\pi_t|^2dt\Big)^{1/2}+3\epsilon E^\alpha\Big(\int_0^\gamma\mathfrak q_t^2|\hat \pi_t|^2dt\Big)^{1/2}\\
\le&3E^\alpha\sup_{t\le\gamma}|\mathfrak p_t|\Big(\int_0^\gamma|\pi_t|^2dt\Big)^{1/2}+3\epsilon E^\alpha\sup_{t\le\gamma}|\mathfrak q_t|\Big(\int_0^\gamma|\hat \pi_t|^2dt\Big)^{1/2}\\
\le&3\Big(E^\alpha\sup_{t\le\gamma}|\mathfrak p_t|^2\Big)^{1/2}\Big(E^\alpha\int_0^\gamma|\pi_t|^2dt\Big)^{1/2}+3\epsilon\Big(E^\alpha\sup_{t\le\gamma}|\mathfrak q_t|^2\Big)^{1/2}\Big(E^\alpha\int_0^\gamma|\hat \pi_t|^2dt\Big)^{1/2}.
\end{align*}
Then we first take $\sfa$ on both sides of the inequality, then let $\epsilon\downarrow0$ and notice that
$$\lim_{\epsilon\downarrow0}\sfa E^\alpha\sup_{t\le\gamma}|\mathfrak p_t|^2=\lim_{\epsilon\downarrow0}\sfa E^\alpha\sup_{t\le\gamma}|\mathfrak q_t|^2=0,$$
therefore the convergence of the ($d+2$)-th component is proved.

For the last component, we recall the function $F^\alpha$ defined in (\ref{F}). From (\ref{hatq}) and (\ref{etad3}), we have
\allowdisplaybreaks\begin{align*}
\frac{ \hat q_t^{\alpha,x+\epsilon\xi}(\epsilon)- 2q_t^{\alpha,x}+\hat q_t^{\alpha,x-\epsilon\xi}(-\epsilon)}{\epsilon^2}=&\int_0^t\bigg[\frac{F^\alpha(\bar z_s^{\alpha,z+\epsilon\xi}(\epsilon))-2F^\alpha(\bar x_s^{\alpha,x})+F^\alpha(\bar z^{\alpha,z-\epsilon\xi}_s(-\epsilon))}{\epsilon^2}\\
&+2r^\alpha_s\frac{F^\alpha(\bar z_s^{\alpha,z+\epsilon\xi}(\epsilon))-F^\alpha(\bar z_s^{\alpha,z-\epsilon\xi}(-\epsilon))}{\epsilon}\\
&+\hat r^\alpha_s\Big(F^\alpha(\bar z_s^{\alpha,z+\epsilon\xi}(\epsilon))+F^\alpha(\bar z_s^{\alpha,z-\epsilon\xi}(-\epsilon))\Big)+o(\epsilon)\bigg]ds\\
\eta_t^{d+3,\alpha}=\int_0^t\Big[F^\alpha_{(\bar\xi_s^{\alpha,\xi})(\bar\xi_s^{\alpha,\xi})}(\bar x_s^{\alpha,x})+&F^\alpha_{(\bar\eta_s^{\alpha,0})}(\bar x_s^{\alpha,x})+4r_s^\alpha F^\alpha_{(\bar\xi_s^{\alpha,\xi})}(\bar x_s^{\alpha,x})+2\hat r_s^\alpha F^\alpha(\bar x_s^{\alpha,x})\Big]ds
\end{align*}
Recall that $|r^\alpha_t|$ and $|\hat r^\alpha_t|$ ($t\le\gamma^\alpha$) is uniformly bounded with respect to $\alpha$. Hence it suffices to show that
\begin{align*}\label{lim55}
\lim_{\epsilon\downarrow0}\bigg(\sup_{\alpha\in\mathfrak A}E\sup_{t\le\gamma}\Big|&\frac{F^\alpha(\bar z_t^{\alpha,z+\epsilon\xi}(\epsilon))-2F^\alpha(\bar x_t^{\alpha,x})+F^\alpha(\bar z^{\alpha,z-\epsilon\xi}_t(-\epsilon))}{\epsilon^2}\\
&-F^\alpha_{(\bar\xi_t^{\alpha,\xi})(\bar\xi_t^{\alpha,\xi})}(\bar x_t^{\alpha,x})-F^\alpha_{(\bar\eta_t^{\alpha,0})}(\bar x_t^{\alpha,x})\Big|\bigg)=0,
\end{align*}
\begin{equation*}\label{lim66}
\begin{gathered}
\lim_{\epsilon\downarrow0}\bigg(\sup_{\alpha\in\mathfrak A}E\sup_{t\le\gamma}\Big|\frac{F^\alpha(\bar z_t^{\alpha,z+\epsilon\xi}(\epsilon))-F^\alpha(\bar z_t^{\alpha,z-\epsilon\xi}(-\epsilon))}{\epsilon}-2F^\alpha_{(\bar\xi_t^{\alpha,\xi})}(\bar x_t^{\alpha,x})\Big|\bigg)=0,
\end{gathered}
\end{equation*}
\begin{equation*}\label{lim77}
\begin{gathered}
\lim_{\epsilon\downarrow0}\bigg(\sup_{\alpha\in\mathfrak A}E\sup_{t\le\gamma}\Big|F^\alpha(\bar z_s^{\alpha,z+\epsilon\xi}(\epsilon))+F^\alpha(\bar z_s^{\alpha,z-\epsilon\xi}(-\epsilon))-2F^\alpha(\bar x_s^{\alpha,x})\Big|\bigg)=0.
\end{gathered}
\end{equation*}
which are true due to (\ref{barlimz}), (\ref{barlimzz}) and Lemma \ref{thm2}. 

We have obtained (\ref{G1}). We next estimate $G_2(\epsilon, T,n)$. From (\ref{V}),
$$V_{(\bar\xi^{\alpha,\xi}_t)(\bar\xi^{\alpha,\xi}_t)}(\bar x^{\alpha,x}_t)+V_{(\bar\eta^{\alpha,0}_t)}(\bar x^{\alpha,x}_t)=e^{-\phi_t^{\alpha,x}}v_{(\xi_t^{\alpha,\xi})(\xi_t^{\alpha,\xi})}(x_t^{\alpha,x})+Y_t^{\alpha,x,\xi,0},$$
where
\begin{align*}
Y_t^{\alpha,x,\xi,0}=&e^{-\phi_t^{\alpha,x}}\Big[v_{(\eta_t^{\alpha,0})}(x_t^{\alpha,x})+2v_{(\xi_t^{\alpha,\xi})}(x_t^{\alpha,x})\big(\xi_t^{d+2,\alpha}-\xi_t^{d+1,\alpha}\big)\\
&+v(x_t^{\alpha,x})\big((\xi_t^{\alpha,\xi})^2-2\xi_t^{d+1}\xi_t^{d+2}+\eta_t^{d+2,\alpha}-\eta_t^{d+1,\alpha}\big)\Big]+\eta_t^{d+3,\alpha}.
\end{align*}
Note that based on our construction of $\hat\pi_t^\alpha$, we have $\eta_t^{d+2,\alpha}=0$. 

It follows that
\begin{align*}
G_2(\epsilon,T,n)=&\sup_{\alpha\in\mathfrak A}E^\alpha_{x,\xi,0}\Big(-V_{(\bar \eta_{\gamma})}(\bar x_{\gamma})-V_{(\bar \xi_{\gamma})(\bar \xi_{\gamma})}(\bar x_{\gamma})\Big)\\
\le& \sup_{\alpha\in\mathfrak A}E^\alpha_{x,\xi} \big(-e^{-\phi_\gamma}v_{(\xi_{\gamma})(\xi_{\gamma})}(x_{\gamma})\big)+\sup_{\alpha\in\mathfrak A}E^\alpha_{x,\xi,0}|Y_{\gamma}|.
\end{align*}
We first claim that
\begin{align}\label{yy}
\sup_{\alpha\in\mathfrak A}E^\alpha_{x,\xi,0}|Y_{\gamma}|\le N\overline{\B}(x,\xi),
\end{align}
where the constant $N$ is independent of $\epsilon, T$ and $n$. Indeed, we recall from Theorem \ref{t1} and Lemma \ref{lemma4} that
\begin{align*}
|v(x)|\le|g|_{0,D}+\psi(x)\sup_{\alpha\in A}|f^\alpha|_{0,D},\ \  |v_{(\xi)}(x)|\le K(|g|_{2,D}+\sup_{\alpha\in A}|f^\alpha|_{1,D})|\xi|.
\end{align*}
Therefore, from the definition of $Y_t$, to prove the estimate (\ref{yy}), it suffices to show that the inequality
$$\sfa E^\alpha\zeta_{\gamma}\le N\overline{\B}(x,\xi)$$
is true if the stochastic process $\zeta^\alpha_t$ is any of the following:
$$|\xi_t^{\alpha,\xi}|^2,\ |\xi_t^{d+1,\alpha}|^2,\ |\xi_t^{d+2,\alpha}|^2,\ |\eta_t^{\alpha,0}|,\ |\eta_t^{d+1,\alpha}|,\ |\eta_t^{d+3,\alpha}|.$$
Applying H\"older inequality, we have
\begin{align*}
\sfa E^\alpha_{\xi}|\xi_\gamma|^2\le&\sfa E^\alpha_{\xi}\sup_{t\le\gamma}|\xi_t|^2\le N\overline{\B}(x,\xi),\\
\sfa E^\alpha|\xi^{d+1}_\gamma|^2\le& K\sfa E^\alpha\int_0^\gamma|r_t|^2dt\le N\overline{\B}(x,\xi),\\
\sfa E^\alpha|\xi^{d+2}_\gamma|^2=&\sfa E^\alpha\int_0^\gamma|\pi_t|^2dt\le N\overline{\B}(x,\xi),\\
\sfa E^\alpha_{0}|\eta_\gamma|\le&\sfa E^\alpha_{0}\sup_{t\le\gamma}|\eta_t|\le N\overline{\B}(x,\xi),\\
\sfa E^\alpha|\eta^{d+1}_\gamma|\le&K\sfa E^\alpha\int_0^\gamma|\hat r_t|dt\le N\overline{\B}(x,\xi).
\end{align*}
It remains to show that
$$\sfa E^\alpha|\eta^{d+3}_\gamma|\le N\overline{\B}(x,\xi).$$
From the definition of $\eta^{d+3,\alpha}_t$, it suffices to show that the inequality
$$\sfa E^\alpha\int_0^\gamma e^{-\phi_t}\zeta_tdt\le N\overline{\B}(x,\xi)$$
is true if the stochastic process $\zeta^\alpha_t$ is any of the following:
\begin{equation*}
\begin{gathered}
|r^\alpha_t|^2,\ |\hat r^\alpha_t|,\ |\xi_t^{\alpha,\xi}|^2,\ |\eta_t^{\alpha,0}|,\ |\xi_t^{d+1,\alpha}\xi_t^{\alpha,\xi}|, |\xi_t^{d+1,\alpha}r_t^\alpha|,\\
|\xi_t^{d+2,\alpha}\xi_t^{\alpha,\xi}|, \ |\xi_t^{d+2,\alpha}r_t^\alpha|, \ |\xi_t^{d+1,\alpha}\xi_t^{d+2,\alpha}|, \ |\eta_t^{d+1,\alpha}|.
\end{gathered}
\end{equation*}
Applying H\"older inequality, we obtain
\allowdisplaybreaks\begin{align*}
\MoveEqLeft\sfa E^\alpha_\xi\int_0^\gamma \big(r_t^2+ |\hat r_t|+|\xi_t|^2\big)dt\\
&\le K\sfa E_\xi^\alpha\int_0^\gamma\Big(|\xi_t|^2+\pxtsops\Big)dt\\
&\le N\overline{\B}(x,\xi),\\
\MoveEqLeft\sfa E^\alpha_0\int_0^\gamma |\eta_t|dt\\
&\le \sfa E^\alpha_0\Big(\int_0^\gamma |\eta_t|^2dt\Big)^{1/2}\\
&\le N\overline{\B}(x,\xi),\\
\MoveEqLeft\sfa E^\alpha_\xi\int_0^\gamma e^{-\phi_t}|\xi_t^{d+1}|\big(|\xi_t|+|r_t|\big)dt\\
&\le\sfa E^\alpha_\xi\Big(\int_0^\gamma e^{-ct}cdt\int_0^\gamma\big(|\xi_t|^2+3r_t^2\big)dt\Big)\\
&\le N\overline{\B}(x,\xi),\\
\MoveEqLeft\sfa E^\alpha_\xi\int_0^\gamma |\xi_t^{d+2}|\big(|\xi_t|+|r_t|\big)dt\\
&\le\sfa E^\alpha_\xi\Big(\sup_{t\le\gamma}|\xi_t^{d+2}|\int_0^\gamma \big(|\xi_t|+|r_t|\big)dt\Big)\\
&\le\sfa \Big(E^\alpha_\xi\sup_{t\le\gamma}|\xi_t^{d+2}|^2\Big)^{1/2}\Big(E^\alpha_\xi\int_0^\gamma \big(|\xi_t|+|r_t|\big)dt\Big)^{1/2}\\
&\le N\overline{\B}(x,\xi),\\
\MoveEqLeft\sfa E^\alpha_\xi\int_0^\gamma e^{-\phi_t}|\xi_t^{d+1}\xi_t^{d+2}|dt\\
&\le\sfa E^\alpha_\xi\Big(\sup_{t\le\gamma}|\xi_t^{d+1}|\int_0^\gamma e^{-\phi_t}|\xi_t^{d+2}|dt\Big)\\
&\le\sfa E^\alpha_\xi\Big(\sup_{t\le\gamma}|\xi_t^{d+1}|\int_0^\gamma e^{-ct}cdt\int_0^\gamma 2|r_t|dt\Big)\\
&\le\sfa \Big(E^\alpha_\xi\sup_{t\le\gamma}|\xi_t^{d+2}|^2\Big)^{1/2}\Big(E^\alpha_\xi\int_0^\gamma 2|r_t|dt\Big)^{1/2}\\
&\le N\overline{\B}(x,\xi),\\
\MoveEqLeft\sfa E^\alpha\int_0^\gamma e^{-\phi_t}|\eta_t^{d+1}|dt\\
&\le\sfa E^\alpha\Big(\int_0^\gamma e^{-ct}cdt\int_0^\gamma2|\hat r_t|dt\Big)\\
&\le N\overline{\B}(x,\xi).
\end{align*}
Gather all these estimates, (\ref{yy}) is proved.

We also claim that
\begin{equation}\label{350}
\begin{gathered}
\varlimsup_{n\uparrow\infty}\varlimsup_{T\uparrow\infty}\varlimsup_{\epsilon\downarrow0}\sup_{\alpha\in\mathfrak A}E^\alpha_{x,\xi} \big(-e^{-\phi_\gamma}v_{(\xi_{\gamma})(\xi_{\gamma})}(x_{\gamma})\big)\\
\le\bigg(\sup_{y\in\partial D_\delta^{\lambda},\zeta\in\Rd\setminus\{0\}}\frac{|v_{(\zeta)(\zeta)}(y)|}{\mathrm B_1(y,\zeta)}+2\bigg)2\overline{\B}(x,\xi).
\end{gathered}
\end{equation}

First, we have
\begin{align*}
\sup_{\alpha\in\mathfrak A}E^\alpha_{x,\xi} \big(-e^{-\phi_\gamma}v_{(\xi_{\gamma})(\xi_{\gamma})}(x_{\gamma})\big)=&\sup_{\alpha\in\mathfrak{A}}E^\alpha_{x,\xi}\frac{(-v)_{(\xi_{\gamma})(\xi_{\gamma})}(x_{\gamma})}{\underline{\mathrm{B}}(x_{\gamma},\xi_{\gamma})}\cdot e^{-\phi_\gamma}\underline{\mathrm{B}}(x_{\gamma},\xi_{\gamma})\\
\le&H_1(\epsilon, T,n)+H_2(\epsilon, T,n),
\end{align*}
where
\begin{align*}
H_1(\epsilon, T,n)=&\sup_{\alpha\in\mathfrak{A}}E^\alpha_{x,\xi}\bigg|\frac{(-v)_{(\xi_{\gamma})(\xi_{\gamma})}(x_{\gamma})}{\underline{\mathrm{B}}(x_{\gamma},\xi_{\gamma})}-\frac{(-v)_{(\xi_{\gamma})(\xi_{\gamma})}(x_{\tau_1})}{\underline{\mathrm{B}}(x_{\tau_1},\xi_{\gamma})}\bigg|\underline{\mathrm{B}}(x_{\gamma},\xi_{\gamma}),\\
H_2(\epsilon, T,n)=&\sup_{\alpha\in\mathfrak{A}}E^\alpha_{x,\xi}\frac{|(-v)_{(\xi_{\gamma})(\xi_{\gamma})}(x_{\tau_1})|}{\underline{\mathrm{B}}(x_{\tau_1},\xi_{\gamma})}\underline{\mathrm{B}}(x_{\gamma},\xi_{\gamma}).
\end{align*}
Then we repeat a similar argument to that of estimating $J_1(\epsilon, T,n)$ and $J_2(\epsilon, T,n)$ in the proof of Theorem \ref{t2}. We should have
$$
\varlimsup_{n\uparrow\infty}\varlimsup_{T\uparrow\infty}\varlimsup_{\epsilon\downarrow0}H_1(\epsilon,T,n)\le4\overline{\B}(x,\xi)$$
and
$$
H_2(\epsilon,T,n)\le\sup_{y\in \partial D_\delta^\lambda,\zeta\in  \Rd\setminus\{0\}}\frac{|(-v)_{(\zeta)(\zeta)}(y)|}{\mathrm{B}_1(y,\zeta)}2\overline{\B}(x,\xi),
$$
which imply (\ref{350}). Combining (\ref{yy}) and (\ref{350}), we obtain
\begin{equation}\label{g2}
\varlimsup_{n\uparrow\infty}\varlimsup_{T\uparrow\infty}\varlimsup_{\epsilon\downarrow0}G_2(\epsilon,T,n)\le\bigg(\sup_{\substack{y\in \{\psi=\delta\}\\\zeta\in  \Rd\setminus\{0\}}}\frac{|(-v)_{(\zeta)(\zeta)}(y)|}{\mathrm{B}_1(y,\zeta)}+N\bigg)\overline{\B}(x,\xi).
\end{equation}
Similar to the last part in the proof of Theorem \ref{t2}, 
we have
\begin{equation*}
\varlimsup_{\delta\downarrow0}\sup_{\substack{x\in\{\psi=\delta\}\\\xi\in\Rd\setminus\{0\}}}\frac{|v_{(\xi)(\xi)}(x)|}{\mathrm B_1(x,\xi)}=\left\{
\begin{array}{ll}
0,&\mbox{ if }\psi_{(\zeta)}(y)\ne0;\\
\displaystyle\frac{|g_{(\zeta)(\zeta)}(y)|}{\lambda^2},&\mbox{ if }\psi_{(\zeta)}(y)=0.
\end{array}
\right.
\end{equation*}
It turns out that for each $x\in D$ and $\xi\in\Rd$,
\begin{align*}
(-v)_{(\xi)(\xi)}(x)\le& N\overline{\B}(x,\xi).
\end{align*}
Consequently, the proof is complete.

\end{proof}

For the proof of the second inequality in (\ref{e3}) and the existence and uniqueness result on (\ref{bellmanae}), see Proof of (2.13) and Proof of the existence and uniqueness of (2.14) in \cite{MR3047001}.

\section{Acknowledgements} The author wishes to express sincere gratitude towards his PhD advisor, Professor Nicolai V. Krylov, for illuminating suggestions and the financial support during the preparation of this paper. The author is also very grateful to Professor Hongjie Dong for inspiring discussions on fully nonlinear PDE theory.


\providecommand{\bysame}{\leavevmode\hbox to3em{\hrulefill}\thinspace}
\providecommand{\MR}{\relax\ifhmode\unskip\space\fi MR }
\providecommand{\MRhref}[2]{%
  \href{http://www.ams.org/mathscinet-getitem?mr=#1}{#2}
}
\providecommand{\href}[2]{#2}

\end{document}